\documentclass[hidelinks,onefignum,onetabnum,final]{siamart220329}  % for arxiv

\usepackage{amsfonts}
\usepackage{graphicx}
\usepackage{epstopdf}
\ifpdf
  \DeclareGraphicsExtensions{.eps,.pdf,.png,.jpg}
\else
  \DeclareGraphicsExtensions{.eps}
\fi

\newsiamremark{remark}{Remark}
\newsiamremark{hypothesis}{Hypothesis}
\crefname{hypothesis}{Hypothesis}{Hypotheses}
\newsiamthm{claim}{Claim}
\newsiamremark{example}{Example}

\usepackage{amsopn}

\usepackage{bm,bbm,empheq,verbatim,fancyvrb,amssymb}
\usepackage{booktabs,multirow,xspace}
\usepackage{pifont}

\usepackage{tikz}
\usetikzlibrary{decorations.pathreplacing}
\usetikzlibrary{graphs,quotes}

\usepackage[kw]{pseudo}
\pseudoset{%
left-margin=6mm,%
topsep=5mm,%
label=\footnotesize\arabic*,%
idfont=\texttt,%
ctfont=\textsl,%
ct-left=\qquad\qquad(,%
ct-right=),%
}

\newcommand{\eps}{\epsilon}
\newcommand{\RR}{\mathbb{R}}

\newcommand{\grad}{\nabla}
\newcommand{\Div}{\nabla\cdot}

\newcommand{\bn}{\mathbf{n}}
\newcommand{\bw}{\mathbf{w}}
\newcommand{\bz}{\mathbf{z}}
\newcommand{\bX}{\mathbf{X}}

\newcommand{\cK}{\mathcal{K}}
\newcommand{\cV}{\mathcal{V}}

\newcommand{\ip}[2]{\left<#1,#2\right>}

\newcommand{\maxR}{R^{\bm{\oplus}}}
\newcommand{\minR}{R^{\bm{\ominus}}}
\newcommand{\iR}{R^{\bullet}}

\newcommand{\rSS}{r_{\text{SS}}}

\newcommand{\XX}{\ding{55}}

\newcommand{\dx}{\, \mathrm{d}x}
\DeclareMathOperator*{\argmin}{arg\,min}

\newfloat{pseudofloat}{t}{xyz}[section]
\floatname{pseudofloat}{Algorithm}

\usepackage{todonotes}

\headers{FASCD for nonlinear variational inequalities}{E. Bueler and P. E. Farrell}

\title{A full approximation scheme multilevel method \\ for nonlinear variational inequalities\thanks{Draft \today.%Submitted to the editors DATE.
\funding{EB was supported by a Faculty Development Travel Award from United Academics, University of Alaska Fairbanks, and thanks Max Heldman for helpful comments. PEF was supported by the Engineering and Physical Sciences Research Council [EPSRC grants EP/R029423/1 and EP/W026163/1]. This work used the ARCHER2 UK National Supercomputing Service \href{https://www.archer2.ac.uk}{www.archer2.ac.uk}). PEF thanks Lawrence Mitchell for advice on implementation, and Jack D.~Betteridge for assistance with running Example \ref{ex:results:sia}.}}}

\author{Ed Bueler\thanks{Department of Mathematics and Statistics, University of Alaska Fairbanks, USA
  (\email{elbueler@alaska.edu}).}
\and Patrick E. Farrell\thanks{Mathematical Institute, University of Oxford, Oxford, UK
  (\email{patrick.farrell@maths.ox.ac.uk}).}}

\graphicspath{{figs/}{fixfigs/}{fixfigs/vcycle/}{genfigs/}}

\begin{document}

\maketitle

\begin{abstract}
We present the \emph{full approximation scheme constraint decomposition} (FASCD) multilevel method for solving variational inequalities (VIs).  FASCD is a common extension of both the full approximation scheme (FAS) multigrid technique for nonlinear partial differential equations, due to A.~Brandt, and the constraint decomposition (CD) method introduced by X.-C.~Tai for VIs arising in optimization.  We extend the CD idea by exploiting the telescoping nature of certain function space subset decompositions arising from multilevel mesh hierarchies.  When a reduced-space (active set) Newton method is applied as a smoother, with work proportional to the number of unknowns on a given mesh level, FASCD V-cycles exhibit nearly mesh-independent convergence rates, and full multigrid cycles are optimal solvers.  The example problems include differential operators which are symmetric linear, nonsymmetric linear, and nonlinear, in unilateral and bilateral VI problems.
\end{abstract}

% REQUIRED

% REQUIRED

\section{Introduction} \label{sec:intro}

The constraint decomposition (CD) methods of Tai \cite{Tai2003} are designed to solve variational inequality (VI) problems arising as the Karush--Kuhn--Tucker optimality conditions for the minimization of a convex functional over a convex set.  In a finite element (FE) context, these methods converge for multilevel and domain decompositions.  The multilevel CD method has almost grid-independent error bounds, and near-optimal complexity, for elliptic, linear obstacle problems \cite[Subsection 5.4]{Tai2003}, \cite[Theorem 4.6]{GraeserKornhuber2009}.

Here we extend Tai's multilevel CD method by replacing the coarser-level corrections with a full approximation scheme (FAS; \cite{Brandt1977,Bruneetal2015}) approach.  Specifically, our \emph{full approximation scheme constraint decomposition} (FASCD) method enjoys the following advantages:
\renewcommand{\labelenumi}{\emph{(\roman{enumi})}}
\begin{enumerate}
\item The FAS coarse corrections are adapted for general VI problems. The method requires neither an objective function nor linearity of the residual.
\item Box constraints, i.e.~both upper and lower obstacles, are allowed.
\item The method is strictly admissible, and so can be applied even if the residual operator is defined only on those functions which satisfy the constraints.
\item V-cycles (Algorithm \ref{alg:fascd}) and full multigrid (FMG; \cite{Trottenbergetal2001}) cycles (Algorithm \ref{alg:fascd-fmg}) are proposed.
\item We identify a ``telescoping'' subset property which can be exploited on the up-smoothing side of a V-cycle for greater performance.
\item The method is (nonlinear) smoother-agnostic.  When a (mesh-dependent) reduced-space Newton method \cite{BensonMunson2006} is used as relaxation, we demonstrate that application of FMG is an optimal solver for the examples in Section \ref{sec:results}.
\end{enumerate}

Prior multilevel algorithms for VIs have limitations compared to the FASCD framework.  The projected full approximation scheme (PFAS) algorithm of Brandt \& Cryer \cite{BrandtCryer1983} applies to linear complementarity problems and it uses projected (truncated) Gauss--Seidel smoothers.  While FASCD reduces to general, nonlinear FAS for PDEs when the inequality constraints are removed, PFAS reduces to linear multigrid.  Another work which combines FAS and VIs is \cite{HintermullerKopacka2011}, but the problems solved, and the approach, are completely different.  For an optimal control problem with an elliptic VI constraint, \cite{HintermullerKopacka2011} applies FAS to a system of PDEs derived from penalizing the VI and the control constraints; their V- and W-cycles are not altered in any manner to maintain admissibility.  On the other hand, the multilevel CD iteration of Tai \cite{Tai2003} is formulated in terms of an objective function, so as stated it only applies to optimization-type problems, and furthermore each of its coarser-level problems refers to the finest-level discretization.  The latter limitation is removed by Gr\"aser \& Kornhuber's \cite[Algorithm 4.2]{GraeserKornhuber2009} reformulation using defect constraints and monotone injection operators; see Section \ref{sec:cdmultilevel} below.  The monotone multigrid methods of Kornhuber and co-authors \cite{GraeserKornhuber2009,Kornhuber1994} appear to be formulated only with Gauss--Seidel type smoothers, and may have been applied only to optimization-type VI problems \cite[for example]{JouvetGraeser2013}.  Furthermore, unlike the fastest variants of monotone multigrid, no zeroing of the basis functions associated with the active set (truncation) is employed in FASCD, and this simplifies the implementation.  Nonetheless FASCD coarse grid corrections can update the estimate of the active set on the fine grid, and indeed the numerical evidence indicates that FASCD coarse corrections are effective even when the coarse meshes are far too coarse to resolve the active set.  On a problem with a complicated coincidence set, for which the truncated Newton/monotone multigrid methods were apparently designed \cite[problem 7.1.1]{GraeserKornhuber2009}, the performance of FASCD is at least as good (Example \ref{ex:results:classical}).  A limitation of FASCD is that, as formulated, it only solves bound-constrained VI problems.

In our experiments, FASCD V-cycles exhibit essentially the same mesh-independent convergence rates for VI problems that FAS multigrid V-cycles exhibit for the corresponding unconstrained PDE problems.  The FMG results in Section \ref{sec:results} show near-textbook multigrid efficiency, yielding solutions within discretization error in work comparable to a few smoother sweeps on the finest-level mesh \cite{BrandtLivne2011}.  Specifically, we see this performance for classical (unilateral, Laplacian) obstacle problems, $p$-Laplacian obstacle problems, and box-constrained advection-diffusion equations.  Furthermore, V-cycle results are largely insensitive to the geometric complexity of the coincidence set and free boundary.

All iterates in FASCD, on all mesh levels, are admissible, and thus the nonlinear operator in the VI needs only be defined for admissible states.  For example, we solve an ice sheet geometry problem (Examples \ref{ex:sia} and \ref{ex:results:sia}) in which the ice flow formulas are only meaningful for surface elevation iterates that do not penetrate the bedrock.  (Non-admissible methods, including semi-smooth methods \cite{BensonMunson2006}, would require unnatural modifications of the operator formula.)  FASCD solves this low-regularity ice sheet problem, wherein the solution has unbounded gradient at the free boundary, by a small, mesh-independent number of FMG iterations (Example \ref{ex:results:sia}). The same Example investigates the parallel weak scaling \cite{Bueler2021} of our implementation.

The paper is organized as follows.  In Section \ref{sec:vi} we recall the theory of coercive and monotone VIs, and describe several motivating examples.  Section \ref{sec:cd} reviews CDs and the iterations they induce.  Section \ref{sec:femultilevel} sets up multilevel FE hierarchies, and then Section \ref{sec:cdmultilevel} extends the CD method via defect constraints generated by monotone injection operators \cite{GraeserKornhuber2009}.  Here we propose an apparently new understanding of telescoping sets for up-smoothing, based on an ``incomplete'' CD and an associated iteration.  Section \ref{sec:vcycle} then states the V-cycle, with Section \ref{sec:implementation} addressing convergence criteria, FMG cycles, and our smoother choice.  Section \ref{sec:results} gives numerical results from our open source Python/Firedrake \cite{Rathgeberetal2016} implementation.

\section{Variational inequalities} \label{sec:vi}

Let $\cV$ be a real reflexive Banach space with norm $\|\cdot\|$ and topological dual space $\cV'$.  Denote the dual pairing of $\phi \in \cV'$ and $v\in\cV$ by $\ip{\phi}{v} = \phi(v)$, and define the (Banach space) norm on $\cV'$ by $\|\phi\|' = \sup_{\|v\|=1} |\!\ip{\phi}{v}\!|$.  Let $\cK \subset \cV$ be a nonempty, closed, and convex subset, the \emph{constraint set}; elements of $\cK$ are said to be \emph{admissible}.  For a continuous, but generally nonlinear, operator $f:\cK \to \cV'$ and a linear \emph{source functional} $\ell\in \cV'$ we consider the following \emph{variational inequality} (VI) problem: find $u^\star\in \cK$ such that
\begin{equation}
\ip{f(u^\star)}{v-u^\star} \ge \ip{\ell}{v-u^\star} \quad \text{for all } v\in \cK. \label{eq:vi}
\end{equation}
Because $f$ is (generally) nonlinear, the source term $\ell$ is not strictly needed in stating this class of problems---by redefining $f$ we may take $\ell=0$---but introducing $\ell$ clarifies the algorithms which follow.

VI \eqref{eq:vi} generalizes nonlinear systems of equations $f(u^\star)=\ell$ to problems where $u^\star$ is also constrained to be in $\cK$.  Pretending the dual pairing is an inner product, by \eqref{eq:vi} the angle between $f(u^\star)-\ell$ and any arbitrary vector $v-u^\star$ pointing from $u^\star$ into $\cK$ is at most $90^\circ$; the dual vector $f(u^\star)-\ell$ must point into $\cK$.  If $u^\star$ is in the interior of the constraint set then \eqref{eq:vi} implies $f(u^\star)=\ell$.
Variational inequalities may generally permit multiple solutions~\cite{Farrell2019}.

The following definitions are standard \cite{KinderlehrerStampacchia1980}.

\begin{definition}  An operator $f:\cK \to \cV'$ is \emph{monotone} if
\begin{equation}
\ip{f(u)-f(v)}{u-v} \ge 0 \qquad \text{for all } u,v \in \cK, \label{eq:monotone}
\end{equation}
\emph{strictly monotone} if equality in \eqref{eq:monotone} implies $u=v$, and \emph{coercive} if there exists $w \in \cK$ so that
\begin{equation}
\frac{\ip{f(u)-f(w)}{u-w}}{\|u-w\|} \to +\infty \qquad \text{as } \|u\|\to +\infty. \label{eq:coercive}
\end{equation}
\end{definition}

It is well-known that if $f:\cK \to \cV'$ is continuous, monotone, and coercive then VI \eqref{eq:vi} has a solution \cite[Corollary III.1.8]{KinderlehrerStampacchia1980}, and that the solution is unique when $f$ is strictly monotone.  As is standard in the calculus of variations \cite{Evans2010}, coercivity permits a compactness argument for unbounded sets $\cK$; recall that the closed and bounded subsets of a reflexive Banach space are weakly compact.  Some VIs solved in this paper satisfy a stronger inequality than \eqref{eq:coercive}.

\begin{definition}  Let $p>1$.  The map $f:\cK \to \cV'$ is \emph{$p$-coercive} if there exists $\kappa>0$ such that
\begin{equation}
\ip{f(u)-f(v)}{u-v} \ge \kappa \|u-v\|^p \qquad \text{for all } u,v \in \cK. \label{eq:pcoercive}
\end{equation}
\end{definition}

It is easy to see that if $f$ is $p$-coercive then it is monotone, strictly monotone, and coercive.  Thus if $f:\cK \to \cV'$ is continuous and $p$-coercive then there exists a unique $u^\star\in \cK$ solving VI \eqref{eq:vi}.

Let $\Omega \subset \RR^d$ denote a bounded, open set with piecewise-smooth boundary.  Sobolev spaces \cite{Evans2010} will be denoted by $W^{k,p}(\Omega)$, for integer $k$ and $1\le p \le \infty$, with $W^{1,p}_0(\Omega)$ denoting the subspace with zero trace.  The following example includes the classical and $p$-Laplacian obstacle problems \cite{ChoeLewis1991}.

\begin{example}  \label{ex:plaplacian}  Let $1<p<\infty$.  For $u,v \in \cV = W^{1,p}_0(\Omega)$, define $f:\cV \to \cV'$ by
\begin{equation}
\ip{f(u)}{v} = \int_\Omega |\grad u|^{p-2} \grad u \cdot \grad v\dx, \label{eq:plaplacian}
\end{equation}
a continuous map \cite[Theorem A.0.6]{Peral1997}.  It follows from inequalities for $|\cdot|^p$ in $\RR^d$, and the Poincar\'e inequality, that $f$ is $p$-coercive \cite{ChoeLewis1991}; \cite[Appendix A]{Bueler2021conservation}.  \end{example}

For monotone operators $f$, VI \eqref{eq:vi} generalizes the problem of minimizing a convex function over $\cK$.  Suppose $F:\cK \to \RR$ is lower semi-continuous and (G\^ateaux) differentiable with continuous derivative $F':\cK \to \cV'$.  Then $F$ is convex if and only if $F'$ is monotone \cite[Proposition I.5.5]{EkelandTemam1976}.  Furthermore, Proposition II.2.1 in \cite{EkelandTemam1976} shows that if $F$ is convex then \eqref{eq:vi} holds for $f=F'$ if and only if
\begin{equation}
u^\star = \argmin_{v\in\cK} F(v) - \ip{\ell}{v}. \label{eq:minimization}
\end{equation}
The CD methods of Tai \cite{Tai2003} address problem \eqref{eq:minimization}; the analysis in \cite{Tai2003} supposes that an objective $F$ exists with $F'$ 2-coercive.  For Example \ref{ex:plaplacian} we may define
\begin{equation}
F(v) = \frac{1}{p} \int_\Omega |\grad v|^p\dx. \label{eq:plaplacianobjective}
\end{equation}
Then $F$ is a convex functional and $F'(v) = f(v)$ for $f$ given in \eqref{eq:plaplacian}.  For any closed and convex $\cK\subset \cV$, VI problem \eqref{eq:vi} is then equivalent to optimization problem \eqref{eq:minimization}.

Next we give two examples which are \emph{not} of optimization type \eqref{eq:minimization}.  The first is an advection-diffusion problem; it requires the following Lemma.

\begin{lemma}  \label{lem:advectionskew}  \cite{Elmanetal2014}\,  Suppose $\bX :\Omega \to \RR^d$ is a bounded and boundedly-differentiable vector field on $\Omega$ with $\Div \bX=0$.  For $u,v \in W^{1,2}(\Omega)$ let $b(u,v) = \int_\Omega (\bX \cdot \grad u) v\dx$.  Then $b(u,u) = \frac{1}{2} \int_{\partial \Omega} u^2 \bX\cdot \bn\dx$ where $\bn$ is the outward normal on $\partial \Omega$.
\end{lemma}

\begin{proof}
Integration by parts gives $b(u,v) = - b(v,u) + \int_{\partial \Omega} uv \bX\cdot \bn\dx$.
\end{proof}

\begin{example}  \label{ex:advectiondiffusion}  Suppose $\partial\Omega$ is partitioned into Dirichlet and Neumann portions, $\partial\Omega = \partial_D\Omega \cup \partial_N\Omega$, with $\partial_D\Omega$ of positive Hausdorff measure.  Let $\cV \subset W^{1,2}(\Omega)$ be the space of functions which have zero trace on $\partial_D\Omega$.  Consider a velocity field $\bX$ on $\Omega$ satisfying the conditions of Lemma \ref{lem:advectionskew}, and assume that the flow is outward on the Neumann boundary: $\bX \cdot \bn \ge 0$ on $\partial_N\Omega$.  For $u,v \in \cV$, $\eps>0$, and $\phi \in \cV'$, define
\begin{equation}
\ip{f(u)}{v} = \eps \left(\grad u, \grad v\right)_{L^2(\Omega)} + b(u,v) - \ip{\phi}{v}. \label{eq:advectiondiffusion}
\end{equation}
With this operator, consider VI \eqref{eq:vi} with $\ell = 0$ for any closed and convex $\cK \subset \cV$. The interior condition in strong form is the following linear advection-diffusion equation
\begin{equation}
-\eps \grad^2 u + \bX\cdot \grad u = \phi.
\label{eq:advectiondiffusionstrong}
\end{equation}
On the other hand, it is easy to see that $|\!\ip{f(u)}{v}\!| \le \left((\eps + \|\bX\|_\infty) \|u\| + \|\phi\|'\right) \|v\|$, so that $f:\cK \to \cV'$ is continuous.  Lemma \ref{lem:advectionskew} says that the bilinear form $b(u,v)$ is skew-symmetric up to a term which is nonnegative by the outward flow assumption.  From the Poincar\'e inequality,
\begin{align*}
\ip{f(u)-f(v)}{u-v} &= \eps \int_\Omega |\grad u - \grad v|^2\dx + b(u-v,u-v) \\
                    &= \eps \int_\Omega |\grad u - \grad v|^2\dx + \frac{1}{2} \int_{\partial_N\Omega} (u-v)^2 \bX\cdot\bn \ge \eps C \|u-v\|^2.
\end{align*}
Thus $f$ is 2-coercive, and so the associated VI problem is well-posed. If $\bX \ne 0$ then $f\ne F'$ for any objective $F$, because $\ip{f(u)}{v}$ does not possess the symmetry of a gradient: for $u,v$ which are zero on $\partial \Omega$, note that $\ip{f(u)}{v} - \ip{f(v)}{u} = -2 b(u,v)$.  References \cite{Bueler2021conservation,ChangNakshatrala2017} consider similar advection-diffusion VI problems over $\cK = \{v\ge 0\}$. \end{example}

The next VI example is a model for steady ice sheet geometry in a given climate.  It is fully nonlinear and not of optimization type.  Related VI problems arise for any fluid layer subject to processes which can remove fluid mass \cite{Bueler2021conservation}.

\begin{example}  \label{ex:sia}  Let $\Omega \subset \RR^2$ be a fixed region of land, and suppose $b \in C^1(\Omega)$ is a given bedrock elevation.  Let $a \in L^2(\Omega)$ denote a given ``surface mass balance'' function, the annually-averaged rate of ice accumulation (snow) minus melt and runoff.  Then the surface elevation $s\in C(\Omega)$ of a steady-state, isothermal, and non-sliding ice sheet (glacier) under the \emph{shallow ice approximation} (SIA) \cite{GreveBlatter2009} satisfies the following VI problem: find $s \in \mathcal{K} = \{s\ge b\}$ such that
\begin{equation}
\int_\Omega \Gamma (s-b)^{n+2} |\grad s|^{n-1} \grad s \cdot \grad (v-s) \dx \ge \int_\Omega a (v-s)\dx \quad \text{ for all } v \in \mathcal{K}, \label{eq:siavi}
\end{equation}
where $n>1$ and $\Gamma>0$ are physical constants for ice; $n=3$ is a typical value \cite{GreveBlatter2009}.  The doubly-nonlinear and doubly-degenerate operator here is
\begin{equation}
\ip{f(s)}{v} = \int_\Omega \Gamma (s-b)^{n+2} |\grad s|^{n-1} \grad s \cdot \grad v\dx. \label{eq:siaoperator}
\end{equation}
If the bed is flat ($b=0$) then a power transformation converts \eqref{eq:siaoperator} into form \eqref{eq:plaplacian} with $p=n+1$.  Less is understood for general beds $b$, but existence holds for the VI \cite{JouvetBueler2012}, and a power of the solution lives in a Sobolev space: $(s-b)^{2p/(p-1)} \in W^{1,p}(\Omega)$.  This theory, and also observations of ice sheets, shows that $|\grad s|$ is usually unbounded as one approaches the free boundary from the icy ($s>b$) side.
\end{example}

The numerical results in Section \ref{sec:results} will apply FASCD Algorithms \ref{alg:fascd} and \ref{alg:fascd-fmg} to the above Examples.

\section{Constraint decomposition (CD)} \label{sec:cd}

Suppose that there are $m<\infty$ closed linear subspaces $\cV_i \subset \cV$, so that the sum
\begin{equation}
\cV = \sum_{i=0}^{m-1} \cV_i \label{eq:subspacedecomp}
\end{equation}
holds in the sense that if $w \in \cV$ then there exist $w_i \in \cV_i$ so that $w = \sum_i w_i$.  We say \eqref{eq:subspacedecomp} is a \emph{space decomposition} of $\cV$ \cite{Xu1992}.  For a closed, convex subset $\cK \subset \cV$, suppose further that $\cK_i \subset \cV_i$ are nonempty, closed, and convex subsets such that
\begin{equation}
\cK = \sum_{i=0}^{m-1} \cK_i. \label{eq:constraintdecomp}
\end{equation}
The sum in \eqref{eq:constraintdecomp} must hold in two senses \cite{TaiTseng2002}: \emph{(i)}~if $w \in \cK$ then there exist $w_i \in \cK_i$ so that $w = \sum_i w_i$, and \emph{(ii)}~if $z_i \in \cK_i$ for each $i$ then $\sum_i z_i \in \cK$.  Note that neither decomposition \eqref{eq:subspacedecomp} or \eqref{eq:constraintdecomp} is required to be unique, and also notice that sense \emph{(ii)} is automatic for \eqref{eq:subspacedecomp} because the $\cV_i$ are subspaces.  Finally, suppose there are bounded, generally nonlinear, \emph{decomposition operators} $\Pi_i : \cK \to \cK_i$ such that if $v \in \cK$ then
\begin{equation}
v = \sum_{i=0}^{m-1} \Pi_i v,  \label{eq:constraintrestrictionsum}
\end{equation}
with the decomposition operators satisfying the stability property
\begin{equation} \label{eq:decompositionstability}
\left(\sum_{i=0}^{m-1} \| \Pi_i u - \Pi_i v \|^2\right)^{1/2} \lesssim \|u - v\|.
\end{equation}
Clearly \eqref{eq:constraintrestrictionsum} implies sense \emph{(i)} for \eqref{eq:constraintdecomp}.  A \emph{constraint decomposition} (CD) of $\cK$ is a choice of $\cV_i,\cK_i,\Pi_i$ satisfying \eqref{eq:subspacedecomp}--\eqref{eq:decompositionstability} \cite{Tai2003}.  Figure \ref{fig:cartoon} suggests how a low-dimensional example might look; note $\cK_i \not\subset \cK$ in this and many other cases.

\begin{figure}[ht]
\centering
\includegraphics[width=0.45\textwidth]{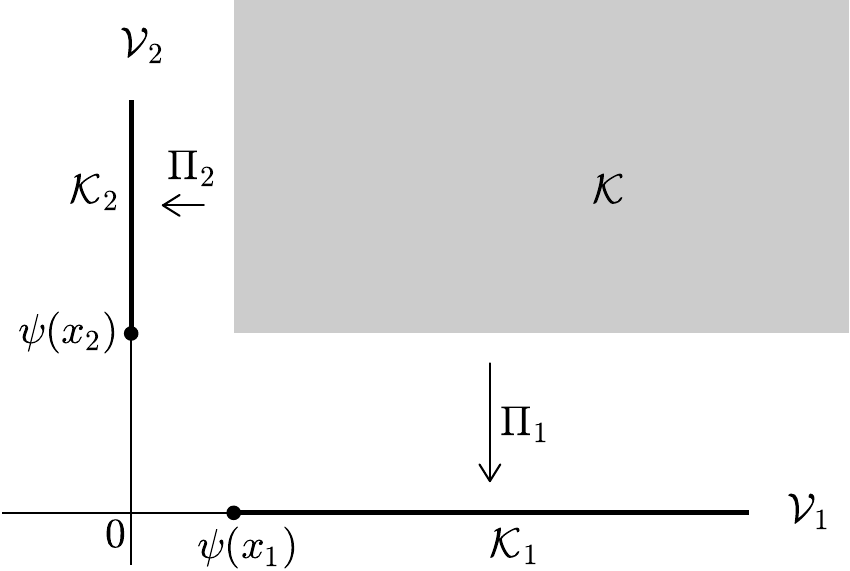}
\caption{A constraint decomposition (CD) for a unilateral obstacle problem on a two-point space $\Omega=\{x_1,x_2\}$, with $\mathcal{V}=\left\{v \,:\, \Omega \to \RR\right\}$ and $\mathcal{K}=\{v\ge \psi\}$.}
\label{fig:cartoon}
\end{figure}

We will construct a multilevel FE CD discretization in Sections \ref{sec:femultilevel}--\ref{sec:implementation}.  However, the CD concept can be applied at the level of the continuum problem.  For example, in a unilateral obstacle problem ($\cK = \{v \ge \psi\}$) one can construct an overlapping domain decomposition of $\cV$ by using a smooth, finite partition of unity $\{\varphi_i\}$, subspaces $\cV_i = \{w \in \cV : w = 0 \text{ if } \varphi_i = 0\}$, constraint sets $\cK_i = \{v \in \cV_i : v \ge \varphi_i \psi\}$, and decomposition operators $\Pi_i(v) = \varphi_i v$.  One may also construct a CD which is a disjoint frequency decomposition of a Hilbert space of periodic functions on a cube $\Omega=(0,1)^d$.  The multilevel CD, over a finite-dimensional FE discretization, will, in a vague sense, approximate such a frequency decomposition.

Associated to any CD are certain iterative methods \cite{Tai2003,Xu1992}.  The following multiplicative/successive \pr{cd-mult} algorithm, which generalizes the Gauss--Seidel iteration, replaces $u \in \cK$ with a new iterate $w\in\cK$ which should be closer to $u^\star \in \cK$ solving \eqref{eq:vi}.  The idea is to solve smaller VI problems over each subset $\cK_i$.  Each subset solution corrects the global iterate immediately, so the ordering of the $\cK_i$ is important.

\begin{pseudo*}
\pr{cd-mult}(u)\text{:} \\+
    for $i = 0,\dots,m-1$: \\+
        \rm{find} $w_i\in \cK_i$ \rm{such that} \\+
            $\displaystyle \Big<f\Big(\sum_{j<i} w_j + w_i + \sum_{j>i} \Pi_j u\Big),\, v_i - w_i\Big> \ge \ip{\ell}{v_i - w_i} \textnormal{ for all } v_i \in \cK_i$ \\--
    return $w=\sum_i w_i\in\cK$
\end{pseudo*}

\noindent The sum ``$\sum_{j>i} \Pi_j u$'' should be read as ``those parts of $u$ which need improvement''.  Tai \cite{Tai2003} also proposes an additive/parallel version of the algorithm, but it is not used here.

The VI problems in \pr{cd-mult} reduce to constrained line searches when the $\cV_i$ are one-dimensional.  If $e_i = w_i - \Pi_i u \in \cV_i$ denotes the correction from $\cK_i$, then the new iterate $w = u + \sum_i e_i$ is computed by adding a correction from each sub\emph{space} $\cV_i$, but so that $w \in \cK$; the corrections preserve admissibility.  Furthermore, if VI problem \eqref{eq:vi} corresponds to convex optimization \eqref{eq:minimization} then \pr{cd-mult} generates monotonically-decreasing objective values.

\begin{lemma} \cite{Tai2003}  Suppose $f=F'$ for $F:\cK\to\RR$ convex and differentiable, and let $u\in\cK$.  The output $w \in \cK$ from \pr{cd-mult} satisfies
\begin{equation}
F(w) - \ip{\ell}{w} \le F(u) - \ip{\ell}{u}.  \label{eq:objectivemonotone}
\end{equation}
\end{lemma}

\pr{cd-mult} is meaningful whether or not we are in the optimization case $f=F'$, and even if $f$ is nonlinear, non-local, or defined only on $\cK$, but its practical and efficient FE implementation for such general operators $f$ seems not to have been addressed in the literature.  In particular, references \cite{GraeserKornhuber2009,Tai2003} only apply the multilevel CD method to the classical obstacle problem, wherein $f=F'$ is linear, local, and defined on all of $\mathcal{V}$.  In Section \ref{sec:vcycle} we will extend multilevel CD iterations to general nonlinear operators by applying the full approximation scheme (FAS) idea of Brandt \cite{Brandt1977}, and then illustrate its application to general operators in Examples \ref{ex:results:plap}--\ref{ex:results:sia}.

From now on we restrict to the case of \emph{box constraints} \cite{BensonMunson2006,FerrisPang1997}, equivalently \emph{bilateral obstacle problems}, over $\mathcal{V}=W^{1,p}(\Omega)$.  The \emph{lower obstacle} $\underline{\gamma}$ and \emph{upper obstacle} $\overline{\gamma}$ are given measurable functions, defined on the closure $\overline{\Omega}$, with values in the extended real line $\tilde\RR = [-\infty,+\infty]$, satisfying the following properties: $\underline{\gamma} < +\infty$, $\overline{\gamma}>-\infty$, and $\underline{\gamma} \le \overline{\gamma}$.  Also suppose $\partial\Omega$ is split into disjoint sets $\partial\Omega = \partial_D \Omega \cup \partial_N \Omega$, with $\partial_D \Omega$ of positive Hausdorff measure, and $\partial_N \Omega$ possibly empty.  Dirichlet data is given by $g_D:\partial_D \Omega \to \RR$, and we assume $\underline{\gamma} \le g_D \le \overline{\gamma}$ on $\partial_D \Omega$.  Now define
\begin{equation}
\cK = \left\{w\,:\,\underline{\gamma} \le w \le \overline{\gamma} \, \text{ and }\, w\big|_{\partial_D \Omega} = g_D\right\} \subset \cV =W^{1,p}(\Omega), \label{eq:originalconstraintset}
\end{equation}
the closed and convex \emph{constraint set}; this treats inequality constraints and Dirichlet conditions in a unified manner.  Dirichlet values are in the trace sense \cite{Evans2010}, and $g_D$ is assumed to be as regular as needed.  Any Neumann conditions over $\partial_N \Omega$ are imposed weakly by the operator $f$, via an integral over $\partial_N\Omega$.  We will numerically solve VI problems \eqref{eq:vi} over constraint sets \eqref{eq:originalconstraintset}.

\section{Multilevel finite elements} \label{sec:femultilevel}

We seek to rapidly compute the approximate solution of box-constrained VI problems \eqref{eq:vi} and \eqref{eq:originalconstraintset} using finite elements (FE).  In this Section a nested multilevel $P_1$ or $Q_1$ element \cite{Elmanetal2014} approximation to \eqref{eq:vi} will be constructed by refinement of a coarse mesh.  The mesh-refinement hierarchy permits the construction of a multilevel CD of the finest-level constraint set.

The mesh hierarchy and FE spaces can be constructed by standard nested refinement.  Let $\Omega \subset \RR^d$ be an open bounded polygon and suppose $\mathcal{T}^0$ is the \emph{coarsest mesh}, a finite set of non-overlapping, nondegenerate cells with union equal to $\overline{\Omega}$ and maximum cell diameter (mesh size) $h_0>0$.  We assume $\mathcal{T}^0$ is conforming, with no hanging nodes.  For $j=1,\dots,J$, let $\mathcal{T}^j$ be the uniform refinement of $\mathcal{T}^{j-1}$; each $\mathcal{T}^{j-1}$ cell becomes $2^d$ cells in $\mathcal{T}^j$ with mesh size $h_j = 2^{-j} h_0$.  We call $\mathcal{T}^J$ the \emph{finest-level mesh} and $J+1$ the \emph{number of levels}.  Let $m_j$ be the number of vertices in $\mathcal{T}^j$, the \emph{degrees of freedom} at the $j$th level.

Let $\mathcal{V}^j$ be nested FE spaces over each mesh $\mathcal{T}^j$, with arbitrary boundary values:
\begin{equation}
\mathcal{V}^0 \subset \mathcal{V}^1 \subset \dots \subset \mathcal{V}^J \subset \mathcal{V}=W^{1,p}(\Omega).  \label{eq:fe:nestedspaces}
\end{equation}
As is common in the discretization of variational inequalities, we employ lowest-order continuous Lagrange elements $P_1$ or $Q_1$ on simplices and tensor-product cells respectively\footnote{Keith \& Surowiec have recently proposed a scheme for solving VIs using high-order finite elements~\cite{Keith2023}.}. This choice is convenient because it enjoys \emph{nodal monotonicity}:
\begin{equation}
\bw \ge \bz \quad \implies \quad w \ge z \label{eq:nodalmonotonicity}
\end{equation}
for all $w,z \in \mathcal{V}^j$, where $\bw,\bz \in \RR^{m_j}$ denote vectors of nodal values. In other words, to enforce an inequality constraint, with $P_1$ or $Q_1$ we only enforce it pointwise on the degrees of freedom. For $k\ge 2$ there are functions in the standard bases of $P_k$ and $Q_k$ which take on negative values \cite[Figure 1.7]{Elmanetal2014}, violating \eqref{eq:nodalmonotonicity}.

To construct box constraints on all meshes in the hierarchy, let $\tilde{\mathcal{V}}^j$ denote the set of functions on the vertices of $\mathcal{T}^j$ with values in $\tilde{\RR} = [-\infty,+\infty]$.  The finest-level \emph{lower} and \emph{upper obstacles} $\underline{\gamma}^J, \overline{\gamma}^J \in \tilde{\mathcal{V}}^J$ are constructed by suitable interpolation of $\underline{\gamma}$, $\overline{\gamma}$, e.g.~by applying the Cl\'ement interpolation operator~\cite{Carstensen2006}.  Again the obstacles must satisfy $\underline{\gamma}^J < +\infty$, $-\infty < \overline{\gamma}^J$, and $\underline{\gamma}^J \le \overline{\gamma}^J$.  On the finest-level mesh the discrete boundary data $g_D^J \in \mathcal{V}^J$ is assumed to satisfy $g_D^J = g_D$ on $\partial_D \Omega$---this restricts $g_D$ in \eqref{eq:originalconstraintset} to be piecewise-linear---and $\underline{\gamma}^J \le g_D^J \le \overline{\gamma}^J$ on those vertices of $\mathcal{T}^J$ which lie in $\partial_D \Omega$.  Then we define
\begin{equation}
\mathcal{K}^J = \big\{w\,:\,\underline{\gamma}^J \le w \le \overline{\gamma}^J \text{ on vertices of } \mathcal{T}^J, \, \text{ and } \, w|_{\partial_D\Omega} = g_D^J|_{\partial_D\Omega}\big\} \subset \mathcal{V}^J. \label{eq:fe:fineconstraintset}
\end{equation}
This nonempty, closed, and convex subset is the finest-level \emph{constraint set}.

Let $f^J:\mathcal{K}^J \to (\mathcal{V}^J)'$ be a discretization of $f$ in \eqref{eq:vi}.  Note that the functional $\ell^J=\ell$ is defined and continuous over $\mathcal{V}^J$.  We seek the FE solution $u^J \in \mathcal{K}^J$ of the finest-level, finite-dimensional, nonlinear VI
\begin{equation}
\ip{f^J(u^J)}{v-u^J} \ge \ip{\ell^J}{v-u^J} \qquad \text{for all } v\in \cK^J. \label{eq:fe:vi}
\end{equation}
If $f^J$ is continuous and $p$-coercive on $\mathcal{K}^J$, in the sense that \eqref{eq:pcoercive} holds for all $u,v \in \mathcal{K}^J$, then \eqref{eq:fe:vi} is well-posed.

While \eqref{eq:fe:vi} approximates \eqref{eq:vi}, we warn the reader that it is usually not conforming.  That is, even if quadrature and other implementation details for $f^J$ are ignored, and despite the vector space inclusion $\mathcal{V}^J \subset \mathcal{V}$, from definitions \eqref{eq:originalconstraintset} and \eqref{eq:fe:fineconstraintset} we do not expect $\mathcal{K}^J \subset \mathcal{K}$ in general.  Even for continuous obstacles $\underline{\gamma}, \overline{\gamma}$, admissibility with respect to the FE obstacles $\underline{\gamma}^J, \overline{\gamma}^J$ will generally not imply $\mathcal{K}$-admissibility.  For a practical example, nodal elevation data in a glacier application misses the topographic detail in-between the vertices \cite{Bueler2016}.

The multilevel algorithm defined in Section \ref{sec:vcycle} requires five transfer operators between meshes.  Of these, three are employed in standard FAS \cite{Trottenbergetal2001}; Table \ref{tab:transfers} will help the reader keep track.  Canonical prolongation $P:\mathcal{V}^{j-1}\to\mathcal{V}^j$ implements the vector space inclusion.  Canonical (dual) restriction $R:(\mathcal{V}^j)'\to(\mathcal{V}^{j-1})'$ acts as an identity in the sense that if $\sigma \in (\mathcal{V}^j)'$ then $\ip{R\sigma}{z} = \ip{\sigma}{z}$ for all $z \in \mathcal{V}^{j-1} \subset \mathcal{V}^j$.  These canonical transfers generally involve a change of FE representation, and in particular $R$ is ``full-weighting'' \cite{Trottenbergetal2001}.  To define injection $\iR:\mathcal{V}^j\to\mathcal{V}^{j-1}$, note that the vertices of $\mathcal{T}^{j-1}$ are vertices of $\mathcal{T}^j$ and so for $z\in\mathcal{V}^j$ define $\iR z$ as the element of $\mathcal{V}^{j-1}$ with the same (coarse) nodal values.  Injection is monotonic in the sense that if $w \le z$ then $\iR w \le \iR z$.

\begin{table}[H]
\centering
\begin{tabular}{llc}
\toprule
\emph{name}  & \emph{mapping}  & \emph{linear?} \\ \midrule
canonical prolongation        & $P:\mathcal{V}^{j-1}\to\mathcal{V}^j$ & \,\checkmark \\
canonical (dual) restriction  & $R:(\mathcal{V}^j)'\to(\mathcal{V}^{j-1})'$ & \,\checkmark \\
(nodal) injection             & $\iR:\mathcal{V}^j\to\mathcal{V}^{j-1}$ & \,\checkmark \\
maximum injection           & $\maxR:\tilde{\mathcal{V}}^j\to\tilde{\mathcal{V}}^{j-1}$ & \ding{55} \\
minimum injection           & $\minR:\tilde{\mathcal{V}}^j\to\tilde{\mathcal{V}}^{j-1}$ & \ding{55} \\
\bottomrule
\end{tabular}
\caption{Transfer operators used in this paper.}
\label{tab:transfers}
\end{table}

The monotone injections $\maxR,\minR$ \cite{Tai2003,GraeserKornhuber2009}, which will only be applied to constraints in $\tilde{\mathcal{V}}^j$, are nonlinear maps which maximize and minimize, respectively, nodal values over the open supports of coarser-mesh basis functions (i.e.~over the stars of the coarse vertices).  To make the construction precise, let $x_q^j$ denote the vertices of mesh $\mathcal{T}^j$, with corresponding basis functions $\psi_q^j$.  For an $\tilde\RR$-valued nodal function $z\in\tilde{\mathcal{V}}^j$ we define functions in $\tilde{\mathcal{V}}^{j-1}$:
\begin{align}
(\maxR z)(x_p^{j-1}) &= \max \{z(x_q^j) \,:\, \psi_p^{j-1}(x_q^j) > 0\}, \label{eq:fe:monotoneinjections} \\
(\minR z)(x_p^{j-1}) &= \min \{z(x_q^j) \,:\, \psi_p^{j-1}(x_q^j) > 0\}. \notag
\end{align}
These operators are constructed to enforce the property
\begin{equation}
\minR z \le z \le \maxR z,  \label{eq:fe:monotoneproperty}
\end{equation}
which does not hold for the ordinary injection operator $\iR$. An illustration of the construction on nested meshes is given in Figure \ref{fig:Rplusminus}; implementing these operators on non-nested meshes could be efficiently done using a supermesh~\cite{Farrell2011}. In the next Section we will use $\maxR,\minR$ to construct ``level defect constraints'' such that the zero function is always admissible.  This construction depends on additional ordering properties which follow from definition \eqref{eq:fe:monotoneinjections}:
\begin{equation}
z\ge 0 \implies \minR z \ge 0 \quad \text{ and } \quad z \le 0 \implies \maxR z \le 0. \label{eq:fe:monotoneadditional}
\end{equation}

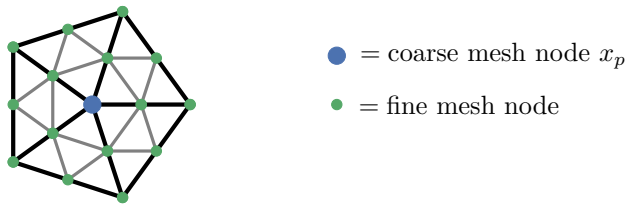
\begin{figure}[ht]
\centering
\definecolor{seabornblue}{rgb}{0.2980392156862745, 0.4470588235294118, 0.6901960784313725}
\definecolor{seaborngreen}{rgb}{0.3333333333333333, 0.6588235294117647, 0.40784313725490196}
\definecolor{seabornred}{rgb}{0.7686274509803922, 0.3058823529411765, 0.3215686274509804}

\begin{tikzpicture}[scale=1.3]
  \foreach \k in {0, 1, 2, 3, 4} {
    %\filldraw[seabornblue, fill=seabornblue] ({cos(deg(2*\k*pi/5))}, {sin(deg(2*\k*pi/5))}) circle (1.5pt);
    \node (A) at ({0.5*cos(deg(2*\k*pi/5)) + 0.5*cos(deg(2*(\k+1)*pi/5))}, {0.5*sin(deg(2*\k*pi/5)) + 0.5*sin(deg(2*(\k+1)*pi/5))}) {};
    \node (B) at ({0.5*cos(deg(2*\k*pi/5))}, {0.5*sin(deg(2*\k*pi/5))}) {};
    \node (C) at ({0.5*cos(deg(2*(\k+1)*pi/5))}, {0.5*sin(deg(2*(\k+1)*pi/5))}) {};
    \node (D) at ({cos(deg(2*\k*pi/5))}, {sin(deg(2*\k*pi/5))}) {};
    \node (E) at ({cos(deg(2*(\k+1)*pi/5))}, {sin(deg(2*(\k+1)*pi/5))}) {};

    \draw [-,gray, very thick] (A.center) -- (B.center);
    \draw [-,gray, very thick] (B.center) -- (C.center);
    \draw [-,gray, very thick] (C.center) -- (A.center);

    \draw [black, ultra thick] (0, 0) -- (D.center);
    \draw [black, ultra thick] (D.center) -- (E.center);

    \filldraw[seaborngreen] (A) circle (1.5pt);
    \filldraw[seaborngreen] (B) circle (1.5pt);
    \filldraw[seaborngreen] (C) circle (1.5pt);
    \filldraw[seaborngreen] (D) circle (1.5pt);
    \filldraw[seaborngreen] (E) circle (1.5pt);

  }
  \filldraw[seabornblue] (0, 0) circle (2.5pt);
  
  \filldraw[seabornblue] (2.5, 0.5) circle (2.5pt);
  \node at (4.1, 0.5) {$= \text{coarse mesh node $x_p$}$};
  \filldraw[seaborngreen] (2.5, 0.0) circle (1.5pt);
  \node (fine) at (3.75, 0.0) {$= \text{fine mesh node}$};
\end{tikzpicture}
\caption{The monotone injection operators $\maxR$ and $\minR$ assign to a coarse degree of freedom (blue) the maximum and minimum respectively of the fine function over the patch of coarse cells sharing its vertex. The coarse mesh is in black and the fine in grey.}
\label{fig:Rplusminus}
\end{figure}

\section{Multilevel CDs from level defect constraints} \label{sec:cdmultilevel}

A key choice in devising a multilevel scheme for VIs is how to construct bounds for the coarser levels.  Our multilevel CD scheme uses \emph{defect constraints} \cite{GraeserKornhuber2009} derived from an admissible finest-level iterate.  The next Example illustrates why we do \emph{not}\footnote{An exception occurs in FMG; see Section \ref{sec:implementation}.} apply injection operators directly to the finest-level box constraints $\underline{\gamma}^J,\overline{\gamma}^J$.

\begin{example}  \label{ex:directRbad}  
Let $\Omega = (0,L) \subset \RR^1$ and consider a two-level mesh hierarchy ($J=1$).  Let $\underline{\gamma}^1$ be a sawtooth function with minimum value zero, and suppose $\overline{\gamma}^1=C+\underline{\gamma}^1$ with $C>0$ strictly less than the maximum of $\underline{\gamma}^1$.  Figure \ref{fig:directRbad} shows such functions; here levels $\mathcal{T}^1$ and $\mathcal{T}^0$ have $5$ and $3$ vertices, respectively.  For appropriate boundary conditions the fine-level constraint set $\mathcal{K}^1$ defined by \eqref{eq:fe:fineconstraintset} is nonempty.  However, various schemes to coarsen the constraints by direct transfers of $\underline{\gamma}^1,\overline{\gamma}^1$ are doomed to failure.  Suppose we propose
\begin{equation}
    \label{eq:fe:badstandard}
    \underline{\gamma}^0, \overline{\gamma}^0 \equiv \iR \underline{\gamma}^1, \iR \overline{\gamma}^1. \qquad \text{\emph{[problematic]}}
\end{equation}
Then $\underline{\gamma}^0=0$ and $\overline{\gamma}^0=C$ identically.  In that case the corresponding constraint set $\mathcal{K}^0$ is nonempty, but prolongation $Pw^0$ of an admissible solution of the coarse problem ($w^0\in\mathcal{K}^0$; $\underline{\gamma}^0 \le w^0 \le \overline{\gamma}^0$) is not admissible on the finer level ($Pw^0 \notin \mathcal{K}^1$).  Attempting to fix this by projecting (truncating) into $\mathcal{K}^1$ would introduce high frequencies not present in $w^0$.  On the other hand, if
\begin{equation}
    \label{eq:fe:badmonotone}
    \underline{\gamma}^0, \overline{\gamma}^0 \equiv \maxR \underline{\gamma}^1, \minR \overline{\gamma}^1 \qquad \text{\emph{[problematic]}}
\end{equation}
then the coarse-level constraint set $\mathcal{K}^0$ is empty because $\underline{\gamma}^0 > \overline{\gamma}^0$ (Figure \ref{fig:directRbad}, right).
\end{example}

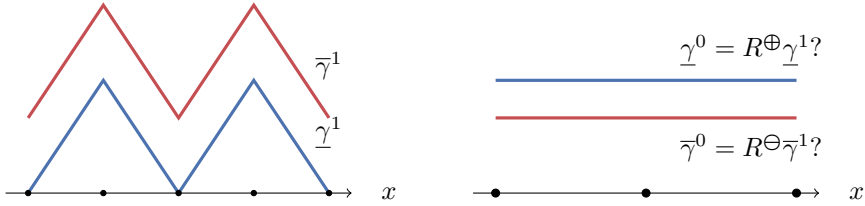
\begin{figure}[ht]
\centering
\begin{tabular}{cc}
\definecolor{seabornblue}{rgb}{0.2980392156862745, 0.4470588235294118, 0.6901960784313725}
\definecolor{seaborngreen}{rgb}{0.3333333333333333, 0.6588235294117647, 0.40784313725490196}
\definecolor{seabornred}{rgb}{0.7686274509803922, 0.3058823529411765, 0.3215686274509804}

\begin{tikzpicture}[scale=1.0]
  \draw[seabornblue, very thick] (0.0,0.0) -- (1.0,1.5) -- (2.0,0.0) -- (3.0,1.5) -- (4.0,0.0)
                     node [yshift=8mm] {{\color{black} $\underline{\gamma}^1$}};
  \draw[seabornred, very thick] (0.0,1.0) -- (1.0,2.5) -- (2.0,1.0) -- (3.0,2.5) -- (4.0,1.0)
                     node [yshift=7mm] {{\color{black} $\overline{\gamma}^1$}};
  \draw[black,thin,->] (-0.3,0.0) -- (4.3,0.0) node [xshift=5mm] {$x$};
  \foreach \x in {0.0, 1.0, 2.0, 3.0, 4.0} {
      \filldraw (\x,0.0) circle (1.0pt);
  }
\end{tikzpicture}
&
\definecolor{seabornblue}{rgb}{0.2980392156862745, 0.4470588235294118, 0.6901960784313725}
\definecolor{seaborngreen}{rgb}{0.3333333333333333, 0.6588235294117647, 0.40784313725490196}
\definecolor{seabornred}{rgb}{0.7686274509803922, 0.3058823529411765, 0.3215686274509804}

\begin{tikzpicture}[scale=1.0]
  \draw[seabornblue,very thick] (0.0,1.5) -- (4.0,1.5)
                     node [xshift=-6mm,yshift=4mm] {{\color{black} $\underline{\gamma}^0 = \maxR\underline{\gamma}^1$?}};
  \draw[seabornred,very thick] (0.0,1.0) -- (4.0,1.0)
                     node [xshift=-6mm,yshift=-4mm] {{\color{black} $\overline{\gamma}^0 = \minR\overline{\gamma}^1$?}};
  \draw[black,thin,->] (-0.3,0.0) -- (4.3,0.0) node [xshift=5mm] {$x$};
  \foreach \x in {0.0, 2.0, 4.0} {
      \filldraw (\x,0.0) circle (1.5pt);
  }
\end{tikzpicture}
\end{tabular}
\caption{For the finer-level constraints $\underline{\gamma}^1$ (blue) and $\overline{\gamma}^1$ (red) on the left, direct application of $\iR,\minR,\maxR$ to generate coarser-level constraints is problematic.  Possibility \eqref{eq:fe:badmonotone}, which generates an empty coarser-level set $\mathcal{K}^0$, is at right.  Defect constraints \eqref{eq:fe:defectconstraints} avoid this difficulty.}
\label{fig:directRbad}
\end{figure}

Iterate-dependent defect constraints will satisfy ordering \eqref{eq:fe:chiordering} below and bypass the difficulties in the above Example.  Suppose $w^J \in \cK^J$ is an admissible iterate for problem \eqref{eq:fe:vi}.  We define the \emph{finest-level defect constraints} \cite{GraeserKornhuber2009} as functions in $\tilde{\mathcal{V}}^J$:
\begin{equation}
\underline{\chi}^J = \underline{\gamma}^J - w^J, \qquad \overline{\chi}^J = \overline{\gamma}^J - w^J. \label{eq:fe:defectconstraints}
\end{equation}
(If $\underline{\gamma}^J=-\infty$ then $\underline{\chi}^J=-\infty$ by definition, and similarly if $\overline{\gamma}^J=+\infty$ then $\overline{\chi}^J=+\infty$.)  Definition \eqref{eq:fe:fineconstraintset} implies that $\underline{\chi}^J \le 0 \le \overline{\chi}^J$ everywhere in $\overline{\Omega}$, including on the Dirichlet boundary $\partial_D\Omega$.    Observe that a corrected iterate $w^J + z^J$ is admissible, $w^J + z^J \in \cK^J$, if and only if the perturbation $z^J \in \mathcal{V}^J$ is zero on $\partial_D\Omega$ and bounded by the defect constraints, $\underline{\chi}^J \le z^J \le \overline{\chi}^J$.

Now descending to coarser levels, $j=J-1,\dots,1,0$, we apply the monotone injections to generate \emph{level defect constraints} (LDCs) in $\tilde{\mathcal{V}}^j$:\footnote{The definition of $\underline{\chi}^j$ matches that in Section 4 of \cite{GraeserKornhuber2009}; see (4.22) and Figure 4.1 therein.}
\begin{equation}
\underline{\chi}^{j} = \maxR \underline{\chi}^{j+1}, \qquad \overline{\chi}^{j} = \minR \overline{\chi}^{j+1}. \label{eq:fe:chilevels}
\end{equation}
Observe that by \eqref{eq:fe:monotoneproperty} and \eqref{eq:fe:monotoneadditional},
\begin{equation}
\underline{\chi}^{J} \le \dots \le \underline{\chi}^0 \le 0 \le \overline{\chi}^0 \le \dots \le \overline{\chi}^J. \label{eq:fe:chiordering}
\end{equation}
The LDC $\underline{\chi}^j$ is the most-negative function from $\tilde{\mathcal{V}}^j$ such that if another element of $\tilde{\mathcal{V}}^j$ is above $\underline{\chi}^j$ then it is also above $\underline{\chi}^{j+1}$; similar comments apply to $\overline{\chi}^{j}$.  Note that for Algorithm \ref{alg:fascd} below, any system satisfying \eqref{eq:fe:chiordering} will suffice; the particular formulas in \eqref{eq:fe:chilevels} are not required.

This construction of multilevel box constraints, likewise the unilateral construction in \cite{GraeserKornhuber2009}, permits large, yet admissible, coarse corrections, which promotes multilevel solver efficiency.  When coarsening the problem downward in a V-cycle, one cannot predict the upcoming coarser corrections and so one must construct constraint sets from which any sum will be admissible.  With this in mind, and looking forward to telescoping sums \eqref{eq:fe:telescoping}, we compute the LDC differences \cite{GraeserKornhuber2009}:
\begin{equation}
\underline{\phi}^j = \underline{\chi}^j - \underline{\chi}^{j-1}, \qquad \overline{\phi}^j = \overline{\chi}^j - \overline{\chi}^{j-1}.  \label{eq:fe:philevels}
\end{equation}
(Take $\underline{\chi}^{-1}=\overline{\chi}^{-1}=0$ so that $\underline{\phi}^0=\underline{\chi}^0$ and $\overline{\phi}^0=\overline{\chi}^0$.  By definition, if $\underline{\chi}^j(x_p)=-\infty$ then we set $\underline{\phi}^j(x_p)=-\infty$ regardless of the value of $\underline{\chi}^{j-1}(x_p)$, and etc.)  Note that while $\underline{\phi}^{j},\overline{\phi}^{j} \in \tilde{\mathcal{V}}^J$ again bracket zero, they are not ordered as in \eqref{eq:fe:chiordering}.

The corresponding \emph{downward} and \emph{upward constraint sets}, respectively, are closed and convex subsets of the FE spaces $\mathcal{V}^j$:
\begin{align}
\mathcal{D}^j = \left\{v \in \mathcal{V}^j \,:\, \underline{\phi}^j \le v \le \overline{\phi}^j \text{ and } \, v|_{\partial_D\Omega} = 0\right\}, \label{eq:fe:constraintsets} \\
\mathcal{U}^j = \left\{v \in \mathcal{V}^j \,:\, \underline{\chi}^j \le v \le \overline{\chi}^j \text{ and } \, v|_{\partial_D\Omega} = 0\right\}. \notag
\end{align}
Because $\underline{\chi}^j \le \underline{\phi}^j \le 0 \le \overline{\phi}^j \le \overline{\chi}^j$, it follows that $\mathcal{D}^j \subseteq \mathcal{U}^j$.  Note that the finest-level set $\mathcal{U}^J$ is equivalent to the original constraint set $\mathcal{K}^J$: for all $z^J \in \mathcal{V}^J$ , $z^J \in \mathcal{U}^J$ if and only if $w^J+z^J \in \mathcal{K}^J$.

The V-cycle solver in the next Section (Algorithm \ref{alg:fascd}) will compute corrections from $\mathcal{D}^j$ and $\mathcal{U}^j$.  For general box constraints the upward direction is, however, based on \emph{incomplete} CDs as we now explain.  Each upward set $\mathcal{U}^j$ is generally only a superset of a sum of downward sets $\mathcal{D}^j$, though this becomes a true CD satisfying \eqref{eq:subspacedecomp}--\eqref{eq:decompositionstability} in unilateral cases.  The following Lemma, apparently new, extends the unilateral CD construction underlying Algorithm 4.2 in \cite{GraeserKornhuber2009}; see also equation (59) in \cite{Tai2003}.

\begin{lemma}  \label{lem:downwardadmissibility}
\begin{equation}
\mathcal{U}^j \supseteq \mathcal{D}^j + \mathcal{D}^{j-1} + \dots + \mathcal{D}^0 \label{eq:fe:downwardsuminclusion}
\end{equation}
for $j=0,1,\dots,J$.  In unilateral cases, where either $\underline{\gamma}^J=-\infty$ identically or $\overline{\gamma}^J=+\infty$ identically, inclusion \eqref{eq:fe:downwardsuminclusion} becomes a full CD with $\mathcal{U}^j=\sum_{i=0}^j \mathcal{D}^i$.
\end{lemma}

\begin{proof}  From nesting \eqref{eq:fe:nestedspaces}, subspace decomposition \eqref{eq:subspacedecomp} holds, and definition \eqref{eq:fe:constraintsets} shows $\mathcal{D}^i \subset \mathcal{U}^i \subset \cV^i$.  Furthermore, telescoping sums hold for any $j=0,1,\dots,J$:
\begin{equation}
\sum_{i=0}^j \underline{\phi}^i = \underline{\chi}^j, \qquad \sum_{i=0}^j \overline{\phi}^i = \overline{\chi}^j.  \label{eq:fe:telescoping}
\end{equation}
Thus if $y^i \in \mathcal{D}^i$ for $0 \le i \le j$ then
\begin{equation}
\underline{\chi}^j = \sum_{i=0}^j \underline{\phi}^i \le \sum_{i=0}^j y^i \le \sum_{i=0}^j \overline{\phi}^i \le \overline{\chi}^j, \label{eq:fe:lemmaordering}
\end{equation}
so \eqref{eq:fe:downwardsuminclusion} holds for any box constraints.

In unilateral cases we can also construct decomposition operators $\Pi_i$ and thereby show \eqref{eq:constraintrestrictionsum}.  For concreteness suppose $\overline{\gamma}^J=+\infty$, thus $\overline{\chi}^j=+\infty$ and $\overline{\phi}^i = +\infty$ for all levels; the $\underline{\gamma}^J=-\infty$ case is similar.  For $v\in \mathcal{V}^j$ and $0\le i \le j$, let $I_{j\to i}^\ominus\,:\,\tilde{\cV}^j \to \tilde{\cV}^i$ be the operator applying the minimum (monotone) injection $j-i$ times: $I_{j\to i}^\ominus v = \minR(\cdots(\minR v)\cdots) $.  (Set $I_{j\to j}^\ominus v = v$ and $I_{j\to -1}^\ominus=0$ by definition.)  Now define nonlinear decomposition operators $\Pi_i:\mathcal{U}^j \to \mathcal{D}^i$:
\begin{equation}
\Pi_i z^j \coloneqq I_{j\to i}^\ominus(z^j - \underline{\chi}^j) - I_{j\to i-1}^\ominus(z^j - \underline{\chi}^j) + \underline{\phi}^i.  \label{eq:fe:unilateraldecompositionoperator}
\end{equation}
(Compare equation (4.9) in \cite{GraeserKornhuber2009}.)  Property \eqref{eq:fe:monotoneproperty} implies $I_{j\to i}^\ominus(z^j - \underline{\chi}^j) \ge I_{j\to i-1}^\ominus(z^j - \underline{\chi}^j)$, thus $\Pi_i z^j \ge \underline{\phi}^i$, so $\Pi_i z^j \in \mathcal{D}^i$, and furthermore stability \eqref{eq:decompositionstability} holds \cite[Theorem 4.2]{GraeserKornhuber2009}.  On the other hand, the following sum telescopes, leaving only its first term plus the sum in \eqref{eq:fe:telescoping}:
\begin{align*}
\sum_{i=0}^j \Pi_i z^j &= z^j - \underline{\chi}^j + \sum_{i=0}^j \underline{\phi}^i = z^j.
\end{align*}
This shows \eqref{eq:constraintrestrictionsum}, thus \eqref{eq:constraintdecomp}, and that \eqref{eq:fe:downwardsuminclusion} is equality.
\end{proof}

One might attempt to convert \eqref{eq:fe:downwardsuminclusion} into a full CD for arbitrary box constraints by constructing different decomposition maps $\Pi_i$, for example by splitting $z^j = \min\{z^j,0\} + \max\{z^j,0\}$ and applying unilateral formulas \eqref{eq:fe:unilateraldecompositionoperator} to $\min\{z^j,0\}$ and $\max\{z^j,0\}$ separately.  The above proof shows that the lower bound $\Pi_i (\min\{z^j,0\}) \ge \underline{\phi}^i$ indeed holds.  However, the following Example shows why $\Pi_i(\min\{z^j,0\})$ may have an arbitrarily large maximum, and thus exceed any finite upper LDC $\overline{\phi}^i$.  In fact, $\Pi_i$ defined above does not generally map all non-positive elements of $\mathcal{U}^J$ into $\mathcal{D}^i$ if a finite upper LDC is present.  The Example does not exclude other strategies for showing full CDs in the general box-constrained case.

\begin{example}  \label{ex:notfullcd}
Let $a > 0$ be an arbitrary positive number.  Consider a 2-level hierarchy ($J=1$), and suppose the lower defect constraint is the constant function $\underline{\chi}^1=-a$.  Note $\underline{\chi}^0=\maxR \underline{\chi}^1=-a$ also, so $\underline{\phi}^1=0$ and $\underline{\phi}^0=-a$.  Suppose $z^1$ is a ``sawtooth'' function which takes on value $-a$ on the coarse vertices and value $0$ on the other vertices.  Then, for any $\overline{\chi}^1\ge 0$, we have $\overline{\chi}^1 \ge 0 \ge z^1\ge \underline{\chi}^1$, so $z^1 \in \mathcal{U}^1$.  However, $I_{1\to 0}^\ominus(z^1 - \underline{\chi}^1) = \minR(z^1 - \underline{\chi}^1) = 0$ identically.  From \eqref{eq:fe:unilateraldecompositionoperator}, $\Pi_1 z^1$ simplifies to $\Pi_1 z^1 = z^1 - \underline{\chi}^1$.  But then $\Pi_1 z^1$ attains a maximum $a\ge 0$.  This maximum can be chosen to exceed any upper defect obstacle $\overline{\phi}^1$, and $\Pi_1 z^1 \notin \mathcal{D}^1$ would then occur.
\end{example}

Each upward set $\mathcal{U}^j$ can also be decomposed, incompletely in the sense of \eqref{eq:fe:downwardsuminclusion}, down to an arbitrary coarser level.  The next Lemma, which does not follow from Lemma \ref{lem:downwardadmissibility}, justifies upward admissibility in Algorithm \ref{alg:fascd}.

\begin{lemma}  \label{lem:upwardadmissibility}  For any $j=0,1,\dots,J$ and $0\le k\le j$,
\begin{equation}
\mathcal{U}^j \supseteq \mathcal{D}^j + \mathcal{D}^{j-1} + \dots + \mathcal{D}^{k+1} + \mathcal{U}^k \label{eq:fe:upwardsuminclusion}
\end{equation}
In unilateral cases, where either $\underline{\gamma}^J=-\infty$ or $\overline{\gamma}^J=+\infty$, actually $\mathcal{U}^j = \mathcal{U}^k + \sum_{i=k+1}^j \mathcal{D}^i$ is a full CD. \end{lemma}

\begin{proof}  Definition \eqref{eq:fe:philevels} shows inclusion \eqref{eq:fe:upwardsuminclusion}.  In unilateral cases we may follow the proof of Lemma \ref{lem:downwardadmissibility} with \eqref{eq:fe:unilateraldecompositionoperator} unchanged for $i=k+1,\dots,j$, but modified in the $i=k$ case: $\Pi_k z^j = I_{j\to k}^\ominus(z^j - \underline{\chi}^j) + \underline{\chi}^k$.  Otherwise the proof goes through as before.
\end{proof}

Inclusions \eqref{eq:fe:downwardsuminclusion} and \eqref{eq:fe:upwardsuminclusion} show the admissibility of all iterates and corrections in Algorithm \ref{alg:fascd} below.  These inclusions are special to multilevel schemes based on defect constraints, and we have seen that they hold regardless of whether decomposition operators $\Pi_i$ can be constructed.  In fact we can propose a new, incomplete variant of the \pr{cd-mult} iteration.  For $w^J\in\mathcal{K}^J$, and existing downward perturbations $y^i\in\mathcal{D}^i$, the iteration applies in the special case where inclusion \eqref{eq:fe:upwardsuminclusion} holds:

\begin{pseudo*}
\pr{icd-tele}(w^J, \{y^i\}_{i=1}^J)\text{:} \\+
    for $j = 0,\dots,J$: \\+
        \rm{find} $z^j\in \mathcal{U}^j$ \rm{so that for all} $v \in \mathcal{U}^j$, \\+
            $\displaystyle \Big<f\Big(w^J + \sum_{i>j} y^i + z^j\Big),v-z^j\Big> \ge \ip{\ell}{v-z^j}$ \\--
    return $\tilde w^J=w^J + z^J$
\end{pseudo*}

\noindent
Note that the downward perturbations $y^i$, which apparently disappear as \pr{icd-tele} proceeds, can actually play a practical role in solving each VI problem, namely when constructing an initial iterate for $z^j$; see the next Section.

Algorithm \ref{alg:fascd} below incorporates \pr{icd-tele} for up-smoothing in the larger constraint sets $\mathcal{U}^j$.  The observation that this can be more efficient, illustrated by Example \ref{ex:results:plap} in particular, seems to be new.  For comparison, $V(1,1)$ cycles in \cite{GraeserKornhuber2009} use half-size versions of the smaller sets $\mathcal{D}^j$ both on descent and ascent; see also \cite[section 5.4]{Tai2003}.  However, any convergence analysis extending the arguments of \cite{GraeserKornhuber2009,Tai2003} might require full CDs.

\section{FASCD V-cycle} \label{sec:vcycle}

The V-cycle in this Section extends the multilevel CD schemes of Tai \cite{Tai2003} and Gr\"aser \& Kornhuber \cite[Algorithm 4.2]{GraeserKornhuber2009} by following a nonlinear full approximation scheme (FAS) approach \cite{BrandtLivne2011}; it reduces to FAS for PDEs when inequality constraints are removed.  Each V-cycle generates a correction to a finest-level iterate $w^J \in \mathcal{K}^J$, via information from each coarser FE subspace $\mathcal{V}^j$, by approximately solving VI problems in constraint sets $\mathcal{D}^j$ and $\mathcal{U}^j$ during the downward and upward stages, respectively.  The corrections are inductively constructed so that iterates remain admissible at every level.  A key idea is that up-smoothing solves less-restricted VI problems than down-smoothing.  Going downward one does not know what coarse corrections are coming, thus one must choose smaller constraint sets ($\mathcal{D}^j$) such that the correction sum is admissible once we return to a given level.  Going upward, however, no further coarse corrections can violate admissibility, so we relax in the larger sets $\mathcal{U}^j$.

Suppose $y^i \in \mathcal{D}^i$, for $i=J,\dots,j+1$, are already-computed corrections during the downward part of the V-cycle (Figure \ref{fig:fascdcycles}, left).  By \eqref{eq:fe:downwardsuminclusion} the correction at this point is admissible, namely $y^J + \dots + y^{j+1} \in \mathcal{U}^J$, equivalently $w^J + y^J + \dots + y^{j+1} \in \mathcal{K}^J$.  Relaxation then occurs at the next-coarser level, yielding the next correction $y^j \in \mathcal{D}^j$.  Solving the coarsest-level VI results in the full downward correction $y^J + \dots + y^1 + z^0 \in \mathcal{U}^J$.  Starting upward, $y^1 + z^0 \in \mathcal{U}^1$ is an initial iterate for relaxation on level $j=1$, which then yields $z^1 \in \mathcal{U}^1$ and an admissible correction $y^J + \dots + y^2 + z^1 \in \mathcal{U}^J$, using inclusion \eqref{eq:fe:upwardsuminclusion}.  Continuing in this ``telescoping'' way, up-smoothing in the constraint sets $\mathcal{U}^j$, starting from initial iterate $y^j+z^{j-1}$, finally generates $z^J\in \mathcal{U}^J$ as the V-cycle correction.  The new finest-level iterate is $\tilde{w}^J = w^J + z^J \in \mathcal{K}^J$.

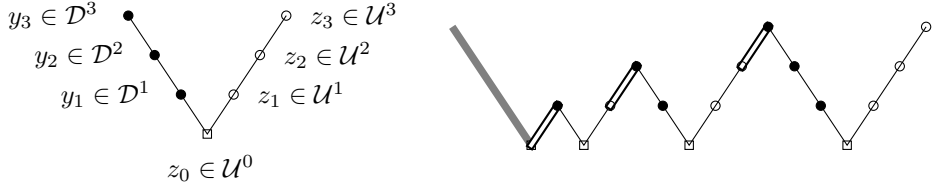
\begin{figure}[ht]
\centering
\begin{tabular}{cc}
\begin{tikzpicture}[scale=0.7]
  \pgfmathsetmacro\hstep{0.5}
  \pgfmathsetmacro\vstep{0.75}
  \pgfmathsetmacro\ceps{0.08}   % size of square for coarse grid

% V-cycle with MCD down-obstacle and up-obstacle annotations
  \draw[black,thin] (-\hstep,3*\vstep) -- (0.0,2*\vstep) -- (\hstep,\vstep) --  (2*\hstep,0.0)
                    -- (3*\hstep,\vstep) -- (4*\hstep,2*\vstep) -- (5*\hstep,3*\vstep);
  \filldraw (-\hstep,3*\vstep) node[xshift=-10mm] {$y_3 \in \mathcal{D}^3$} circle (2.5pt);
  \filldraw (0.0,2*\vstep) node[xshift=-10mm] {$y_2 \in \mathcal{D}^2$} circle (2.5pt);
  \filldraw (\hstep,\vstep) node[xshift=-10mm] {$y_1 \in \mathcal{D}^1$} circle (2.5pt);
  \draw     (2*\hstep-\ceps,-\ceps) node[xshift=1mm,yshift=-4mm] {$z_0 \in \mathcal{U}^0$} rectangle (2*\hstep+\ceps,+\ceps);
  \draw     (3*\hstep,\vstep) node[xshift=9mm] {$z_1 \in \mathcal{U}^1$} circle (2.5pt);
  \draw     (4*\hstep,2*\vstep) node[xshift=9mm] {$z_2 \in \mathcal{U}^2$} circle (2.5pt);
  \draw     (5*\hstep,3*\vstep) node[xshift=9mm] {$z_3 \in \mathcal{U}^3$} circle (2.5pt);
\end{tikzpicture}
&
\begin{tikzpicture}[scale=0.7]
  \pgfmathsetmacro\hstep{0.5}
  \pgfmathsetmacro\vstep{0.75}
  \pgfmathsetmacro\ceps{0.08}   % size of square for coarse grid

% initial restriction to coarsest
  \pgfmathsetmacro\hoff{0*\hstep}
  \draw[shift={(\hoff,0)},gray,line width=1.0mm] (0.0,3*\vstep) -- (\hstep,2*\vstep) --  (2*\hstep,\vstep) -- (3*\hstep,0.0);
  \draw[shift={(\hoff,0)}]     (3*\hstep-\ceps,-\ceps) rectangle (3*\hstep+\ceps,+\ceps);
  \draw[shift={(\hoff,0)},black,thick,double,double distance between line centers=0.8mm,line cap=rect] (3*\hstep,0.0) -- (4*\hstep,\vstep);

% V-cycle to level 1
  \pgfmathsetmacro\hoff{4*\hstep}
  \draw[shift={(\hoff,0)},black,thin] (0.0,\vstep) -- (\hstep,0.0) -- (2*\hstep,\vstep);
  \draw[shift={(\hoff,0)},black,thick,double,double distance between line centers=0.8mm,line cap=rect] (2*\hstep,\vstep) -- (3*\hstep,2*\vstep);
  \filldraw[shift={(\hoff,0)}] (0.0,\vstep) circle (2.5pt);
  \draw[shift={(\hoff,0)}]     (\hstep-\ceps,-\ceps) rectangle (\hstep+\ceps,+\ceps);
  \draw[shift={(\hoff,0)}]     (2*\hstep,\vstep) circle (2.5pt);

% V-cycle to level 2
  \pgfmathsetmacro\hoff{7*\hstep}
  \draw[shift={(\hoff,0)},black,thin] (0.0,2*\vstep) --  (\hstep,\vstep) -- (2*\hstep,0.0) -- (3*\hstep,\vstep) -- (4*\hstep,2*\vstep);
  \draw[shift={(\hoff,0)},black,thick,double,double distance between line centers=0.8mm,line cap=rect] (4*\hstep,2*\vstep) -- (5*\hstep,3*\vstep);
  \filldraw[shift={(\hoff,0)}] (0.0,2*\vstep) circle (2.5pt);
  \filldraw[shift={(\hoff,0)}] (\hstep,\vstep) circle (2.5pt);
  \draw[shift={(\hoff,0)}]     (2*\hstep-\ceps,-\ceps) rectangle (2*\hstep+\ceps,+\ceps);
  \draw[shift={(\hoff,0)}]     (3*\hstep,\vstep) circle (2.5pt);
  \draw[shift={(\hoff,0)}]     (4*\hstep,2*\vstep) circle (2.5pt);

% V-cycle to finest (level 3)
  \pgfmathsetmacro\hoff{12*\hstep}
  \draw[shift={(\hoff,0)},black,thin] (0.0,3*\vstep) -- (\hstep,2*\vstep) --  (2*\hstep,\vstep) -- (3*\hstep,0.0) -- (4*\hstep,\vstep) -- (5*\hstep,2*\vstep) -- (6*\hstep,3*\vstep);
  \filldraw[shift={(\hoff,0)}] (0.0,3*\vstep) circle (2.5pt);
  \filldraw[shift={(\hoff,0)}] (\hstep,2*\vstep) circle (2.5pt);
  \filldraw[shift={(\hoff,0)}] (2*\hstep,\vstep) circle (2.5pt);
  \draw[shift={(\hoff,0)}]     (3*\hstep-\ceps,-\ceps) rectangle (3*\hstep+\ceps,+\ceps);
  \draw[shift={(\hoff,0)}]     (4*\hstep,\vstep) circle (2.5pt);
  \draw[shift={(\hoff,0)}]     (5*\hstep,2*\vstep) circle (2.5pt);
  \draw[shift={(\hoff,0)}]     (6*\hstep,3*\vstep) circle (2.5pt);

  \draw[white] (0, -\vstep) circle (2.5pt);
\end{tikzpicture}
\end{tabular}
\caption{Left: The FASCD V-cycle Algorithm \ref{alg:fascd} computes downward corrections $y_j \in \mathcal{D}^j$, but the upward corrections $z_j\in\mathcal{U}^j$ are in larger sets. Right: FMG Algorithm \ref{alg:fascd-fmg} generates the initial iterate on finer levels by injection and truncation (doubled edges).}
\label{fig:fascdcycles}
\end{figure}

We must specify the problems solved at each mesh level.  As in FAS multigrid for nonlinear PDEs \cite{BrandtLivne2011,Bruneetal2015,Trottenbergetal2001}, both the solution approximation and the residual must be coarsened, and a source functional created.  On mesh level $j$ we assume a re-discretized nonlinear operator $f^j$, with output in $(\cV^j)'$, which is an FE discretization of $f$ in \eqref{eq:vi}.  At the start of each V-cycle we assume that an admissible iterate $w^J \in \mathcal{K}^J$ has been used to generate the LDCs and constraint sets $\mathcal{D}^j$, $\mathcal{U}^j$ at each level (Section \ref{sec:cdmultilevel}); they do not depend on the coarse-level corrections.

As we descend we define the initial $j$th-level iterates
\begin{equation}
w^j = \iR(w^{j+1} + y^{j+1}),  \label{eq:fe:definew}
\end{equation}
from which we define $\ell^j \in (\cV^j)'$:
\begin{equation}
\ell^j = \begin{cases} \ell^J, & j=J \\
                       f^j(w^j) + R\left(\ell^{j+1}-f^{j+1}(w^{j+1}+y^{j+1})\right), & j<J. \end{cases} \label{eq:fe:levelsource}
\end{equation}
As is typical of FAS \cite[section 5.3]{Trottenbergetal2001}, injection $\iR$ coarsens the (primal) iterates, but canonical restriction $R$ coarsens (dual) linear functionals.  Descending through the levels $j=J,J-1,\dots,1$ we solve the following VI problems for $y^j \in \mathcal{D}^j$:
\begin{equation}
\ip{f^j(w^j + y^j)}{v-y^j} \ge \ip{\ell^j}{v-y^j} \qquad \text{for all } v\in \mathcal{D}^j. \label{eq:fe:downvi}
\end{equation}
The initial iterate is $y_{(0)}^j=0$; this is admissible because $\underline{\phi}^j \le 0 \le \overline{\phi}^j$.

On ascent, $j=0,1,\dots,J$, the same VI is solved but for $z^j \in \mathcal{U}^j$:
\begin{equation}
\ip{f^j(w^j + z^j)}{v-z^j} \ge \ip{\ell^j}{v-z^j} \qquad \text{for all } v\in \mathcal{U}^j. \label{eq:fe:upvi}
\end{equation}
However, the nontrivial initial iterate for this problem is found by prolonging $z^{j-1} \in \mathcal{U}^{j-1}$ and then updating the current (i.e.~down-smoothed) correction:
\begin{equation}
z_{(0)}^j = y^j + P z^{j-1}.  \label{eq:fe:upwardinitial}
\end{equation}
An exception is that $z_{(0)}^0=0$ on the coarsest level.  Note $z_{(0)}^j$ is admissible by \eqref{eq:fe:upwardsuminclusion}.

Formula \eqref{eq:fe:levelsource} for $\ell^j$ is a classical FAS construction; compare equation (8.5b) in \cite{BrandtLivne2011} or equation (5.3.14) in \cite{Trottenbergetal2001}.  It has the following explanation, inductively on VI problems \eqref{eq:fe:downvi}.  Because the down-smoother was only approximate, problem \eqref{eq:fe:downvi} was not solved exactly by the finer-level correction $y^{j+1} \in \mathcal{D}^{j+1}$.  Suppose we seek an improved correction by adding $y \in \mathcal{D}^{j}$, to give $y^{j+1}+y \in \mathcal{D}^{j+1}+\mathcal{D}^j \subset \cV^{j+1}$, so that
\begin{equation}
\ip{f^{j+1}(w^{j+1}+y^{j+1}+y)}{v-y} \ge \ip{\ell^{j+1}}{v-y} \quad \text{\emph{[notional]}} \label{eq:fe:downvinotional}
\end{equation}
for all $v \in \mathcal{D}^j$.  (Note $v-y = (y^{j+1}+v)-(y^{j+1}+y)$ and that nesting $\cV^j \subset \cV^{j+1}$ is assumed here.)

However, because the down-smoother has made progress, $y$ should be approximate-able by solving a coarser, $j$th-level problem generated by modifying \eqref{eq:fe:downvinotional} in 4 steps: \emph{(i)} subtract the computable quantity $f^{j+1}(w^{j+1}+y^{j+1}) \in (\mathcal{V}^{j+1})'$ from both sides, \emph{(ii)} replace the relaxed residual $\ell^{j+1}-f^{j+1}(w^{j+1}+y^{j+1})$ on the right by its restriction, \emph{(iii)} replace $w^{j+1}+y^{j+1}$, where it appears on the left, by its injection, and \emph{(iv)} where it appears on the left, replace $f^{j+1}$ by the coarser rediscretization $f^j$.  These steps yield a problem for the $j$th-level correction $y=y^j \in \mathcal{D}^j$:
\begin{align}
&\ip{f^j\left(\iR(w^{j+1}+y^{j+1})+y^j\right)}{v-y^j} - \ip{f^j\left(\iR(w^{j+1}+y^{j+1})\right)}{v-y^j} \label{eq:fe:downvicluttered} \\
&\qquad \ge \ip{R\left(\ell^{j+1}-f^{j+1}(w^{j+1}+y^{j+1})\right)}{v-y^j}, \notag
\end{align}
for all $v\in \mathcal{D}^{j}$.  Though in cluttered form, this is VI problem \eqref{eq:fe:downvi}, once definitions \eqref{eq:fe:definew} and \eqref{eq:fe:levelsource} are applied.

The explanation for \eqref{eq:fe:upvi} and \eqref{eq:fe:upwardinitial} is simpler.  Whether or not problem \eqref{eq:fe:downvi} has been exactly solved by the computed $y^j\in\mathcal{D}^j$, when going upward we can expand problem \eqref{eq:fe:downvi} into the larger set $\mathcal{U}^j \supset \mathcal{D}^j$, and \eqref{eq:fe:upwardsuminclusion} assures admissibly upward from any approximate solution to \eqref{eq:fe:upvi}.  Generally $y^j$ is not the solution to \eqref{eq:fe:upvi}; e.g.~in the optimization case $f=F'$ it would not be optimal in the enlarged constraint set.  Our current best estimate of the solution to \eqref{eq:fe:upvi} is $z^j \approx y^j + Pz^{j-1}\in \mathcal{U}^j$, so this is the initial iterate for relaxation.

These ideas come together in Algorithm \ref{alg:fascd}, \pr{fascd-vcycle}, which updates $w^J$.  The \pr{smooth} and \pr{solve} procedures are assumed to do in-place modifications of their final arguments, without modifying the others.

\begin{pseudofloat}[ht]
\begin{pseudo}
\pr{fascd-vcycle}(J,\ell^J,\underline{\gamma}^J,\overline{\gamma}^J;w^J)\text{:} \\+
    $\underline{\chi}^J, \, \overline{\chi}^J = \underline{\gamma}^J - w^J, \, \overline{\gamma}^J - w^J {\large \strut}$ \label{line:vcyclegenchifinest} \\
    for $j=J$ downto $j=1$ \\+
      $\underline{\chi}^{j-1}, \, \overline{\chi}^{j-1} = \maxR \underline{\chi}^j, \, \minR \overline{\chi}^j {\large \strut}$ \label{line:vcyclegenchi} \\
      $\underline{\phi}^j, \, \overline{\phi}^j = \underline{\chi}^j - P\underline{\chi}^{j-1}, \, \overline{\chi}^j - P\overline{\chi}^{j-1} {\large \strut}$ \label{line:vcyclegenphi} \\
      $y^j = 0$ \\
      $\text{\pr{smooth}}^{\text{\id{down}}}(\ell^j,\underline{\phi}^j,\overline{\phi}^j,w^j;y^j)$  \ct{problem \eqref{eq:fe:downvi} in $\mathcal{D}^j$}\\
      $w^{j-1} = \iR(w^j + y^j)$ \label{line:vcyclerestrictsolution} \\
      $\ell^{j-1} = f^{j-1}(w^{j-1}) + R \left(\ell^j - f^j(w^j+y^j)\right)$ \label{line:vcyclerestrictell} \\-
    $z^0 = 0$ \\
    $\text{\pr{solve}}(\ell^0,\underline{\chi}^0,\overline{\chi}^0,w^0;z^0)$ \hspace{1.0cm} \ct{problem \eqref{eq:fe:upvi} in $\mathcal{U}^0$} \\
    for $j=1$ to $j=J$ \\+
      $z^j = y^{j} + P z^{j-1}$ \label{line:vcycleupsmoothinitial} \\
      $\text{\pr{smooth}}^{\text{\id{up}}}(\ell^j,\underline{\chi}^j,\overline{\chi}^j,w^j;z^j)$  \ct{problem \eqref{eq:fe:upvi} in $\mathcal{U}^j$} \\-
    return $w^J+z^J$
\end{pseudo}
\caption{The FASCD V-cycle, an iteration for solving VI problem \eqref{eq:fe:vi}.}
\label{alg:fascd}
\end{pseudofloat}

With \id{down} $=1$ and \id{up} $=1$ smoother iterations on each level, Algorithm \ref{alg:fascd} is a $V(1,1)$ cycle.  The $V(1,0)$ cycle with \id{up} $=0$ is similar to Algorithm 4.2 in \cite{GraeserKornhuber2009}, but here it is generalized by applying FAS-type corrections and allowing any smoother.  In unilateral cases one $V(1,1)$ iteration of \pr{fascd-vcycle} approximates applications of \pr{cd-mult} (Section \ref{sec:cd}) over CD \eqref{eq:fe:downwardsuminclusion}, followed by \pr{icd-tele} (Section \ref{sec:cdmultilevel}) over \eqref{eq:fe:upwardsuminclusion}, and in this sense $V(1,1)=V(1,0)+V(0,1)$.

Regarding the storage required, if all box constraints are nontrivial then 7 vectors must be allocated on each level: $\underline{\chi}^j$, $\overline{\chi}^j$, $\underline{\phi}^j$, $\overline{\phi}^j$, $w^j$, $\ell^j$, $y^j$; note vectors $z^j$ and $y^j$ may use the same storage.  On the finest level one must also store $\underline{\gamma}^J$ and $\overline{\gamma}^J$.  Therefore, using $m_j=2^d m_{j-1}$ and a geometric series argument, the total storage is at most
\begin{equation}
9 m_J + 7 m_{J-1} + \dots + 7 m_1 + 7 m_0 \le \left(9 + \frac{7}{2^d - 1}\right) m_J,
\end{equation}
namely $16m_J$, $12m_J$, and $10m_J$ in dimensions $d=1,2,3$, respectively.  For comparison, a single-level method needs $4 m_J$ storage (i.e.~$\underline{\gamma}^J,\overline{\gamma}^J,\ell^J,w^J$).  In unilateral cases the storage can be reduced at all levels.

\section{Implementation} \label{sec:implementation}

In this section we propose a convergence criterion, construct the full multigrid (FMG) extension of Algorithm \ref{alg:fascd}, and summarize the reduce-space Newton method used as the smoother and coarse solver.

Repeated application of Algorithm \ref{alg:fascd} to solve finite-dimensional VI problem \eqref{eq:fe:vi} will not generally reduce the residual $f^J(w^J) - \ell^J$ to zero everywhere.  However, the VI is equivalent to a \emph{mixed complementarity problem} (MCP) \cite{FacchineiPang2003}, which says, by definition, that certain conditions hold for $u^J$ at nodes $x_p \in \mathcal{T}^J$.  Let $r(u^J)_p = \ip{f^J(u^J)-\ell^J}{\psi_p}$ be the ordinary (pointwise) residual.  If $x_p \notin \partial_D\Omega$ then exactly one of these four conditions must hold:
\begin{itemize}
\item $\underline{\gamma}^J(x_p)<u^J(x_p)<\overline{\gamma}^J(x_p)$ and $r(u^J)_p = 0$,
\item $\underline{\gamma}^J(x_p)=u^J(x_p)<\overline{\gamma}^J(x_p)$ and $r(u^J)_p \ge 0$,
\item $\underline{\gamma}^J(x_p)<u^J(x_p)=\overline{\gamma}^J(x_p)$ and $r(u^J)_p \le 0$, or
\item $\underline{\gamma}^J(x_p)=u^J(x_p)=\overline{\gamma}^J(x_p)$.
\end{itemize}
Otherwise, if $x_p \in \partial_D\Omega$ then
\begin{itemize}
\item $u^J(x_p)=g_D^J(x_p)$.
\end{itemize}
These 5 cases might be labeled \emph{inactive}, \emph{lower active}, \emph{upper active}, \emph{pinched} (\emph{both active}), and \emph{boundary}, respectively.  Where \emph{pinched} or \emph{boundary} apply, $r(u^J)_p$ can have any value.

This MCP interpretation allows us to construct expressions whose norms should approach zero.  Define the Fischer-Burmeister function \cite{BensonMunson2006,Ulbrich2011} on $a,b\in\RR$:
\begin{equation}
\phi(a,b) = a + b - \sqrt{a^2 + b^2}. \label{eq:phiFB}
\end{equation}
Observe that $\phi$ is semi-smooth, thus continuous, and that $\phi(a,b)=0$ if and only if $a,b$ are complementary, i.e.~$a\ge 0$, $b\ge 0$, and $ab=0$.  For any $w^J\in \mathcal{K}^J$ denote $\underline{g}_p = w^J(x_p) - \underline{\gamma}^J(x_p)$ and $\overline{g}_p = \overline{\gamma}^J(x_p) - w^J(x_p)$, the nonnegative \emph{gaps} at $x_p$.  Now define the \emph{semi-smooth residual}
\begin{equation}
\rSS(w^J)_p = \begin{cases}
\max\big\{\phi\big(\underline{g}_p, r(w^J)_p\big), \phi\left(\overline{g}_p, -r(w^J)_p\right)\big\}, & x_p \notin \partial_D\Omega, \\
w^J(x_p) - g_D^J(x_p), & x_p \in \partial_D\Omega.
\end{cases} \label{eq:rSS}
\end{equation}
Straightforward modifications apply for unilaterally-constrained or unconstrained degrees of freedom, where $\underline{\gamma}^J(x_p) = -\infty$, $\overline{\gamma}^J(x_p) = +\infty$, or both.  VI problem \eqref{eq:fe:vi} is exactly solved by $u^J\in\cK^J$ if and only if $\rSS(u^J)=0$ identically.

Suppose that Algorithm \ref{alg:fascd} is started with $w^{J,0} \in \mathcal{K}^J$ and generates (fine-level) iterates $w^{J,k} \in \mathcal{K}^J$.  Let $\id{atol}$, $\id{rtol}$, $\id{stol} \ge 0$ be given absolute, relative, and step tolerances, respectively.  We employ
\begin{equation}
\|\rSS(w^{J,k})\|_2 < \id{atol} \quad \text{or} \quad \frac{\|\rSS(w^{J,k})\|_2}{\|\rSS(w^{J,0})\|_2} < \id{rtol} \quad \text{or} \quad \frac{\|\Delta w^{J,k}\|_{L^2}}{\|w^{J,k}\|_{L^2}} < \id{stol}, \label{eq:stoppingcriterion}
\end{equation}
where $\Delta w^{J,k} = w^{J,k} -  w^{J,k-1}$, $\|\cdot\|_2$ is the euclidean norm on $\RR^{m_J}$, and $\|\cdot\|_{L^2}$ is the norm on $L^2(\Omega)$, as the stopping criterion.  Defaults in our Firedrake implementation match those of PETSc SNES \cite{Balayetal2023}: $\id{atol}=10^{-50}$, $\id{rtol}=\id{stol}=10^{-8}$; see Section \ref{sec:results} for specific settings in Examples.

In Section \ref{sec:results} we will show that V-cycle iterations alone, starting with some finest-mesh initial iterate $w^{J,0}$ and repeated until \eqref{eq:stoppingcriterion} holds, serve as an effective solver.  However, an FMG cycle, Algorithm \ref{alg:fascd-fmg} below and Figure \ref{fig:fascdcycles} (right), is even better.  The well-known principle \cite[section 2.6]{Trottenbergetal2001} is to start on the coarsest mesh and solve the problem from any initial iterate $w^0\in\mathcal{K}^0$; one may generate it by injection and projection/truncation of $w^{J,0}$.  Then an admissible initial iterate on the next-finer mesh is generated via prolongation.  However, in the VI context one must also project (truncate) to the bounds after the prolongation; see lines \ref{line:fcycleprolongtruncone} and \ref{line:fcycleprolongtrunctwo} in Algorithm \ref{alg:fascd-fmg}.  During this ``ramp'' stage one approximately solves on each level using a fixed number (\id{rampv}) of V-cycles, each down to the coarsest mesh.  When the ramp is complete, one does full-depth V-cycles to convergence.

\begin{pseudofloat}[h]
\begin{pseudo}
\pr{fascd-fmg}(J,\ell^J,\underline{\gamma}^J,\overline{\gamma}^J)\text{:} \\+
    for $j=J$ downto $j=1$ \\+
        $\ell^{j-1} = R(\ell^j)$ \\
        $\underline{\gamma}^{j-1}, \, \overline{\gamma}^{j-1} = \iR \underline{\gamma}^{j}, \, \iR \overline{\gamma}^{j}$ {\large\strut} \\-
    choose $w^0 \in \mathcal{K}^0$ \label{line:fcyclecoarsestinitial} \\
    $\text{\pr{solve}}(\ell^0,\underline{\gamma}^0,\overline{\gamma}^0;w^0)$ \\
    for $j=1$ to $j=J-1$ \\+
        $w^j = \max\left\{\underline{\gamma}^{j},\min\{\overline{\gamma}^{j}, Pw^{j-1}\}\right\} \in \mathcal{K}^j$ \label{line:fcycleprolongtruncone} \\
        $\pr{fascd-vcycle}^{\text{\id{rampv}}}(j,\ell^j,\underline{\gamma}^j,\overline{\gamma}^j;w^j)$ \\-
    $w^J = \max\left\{\underline{\gamma}^{J},\min\{\overline{\gamma}^{J}, Pw^{J-1}\}\right\} \in \mathcal{K}^J$ \label{line:fcycleprolongtrunctwo} \\
    while not \eqref{eq:stoppingcriterion} \\+
        \pr{fascd-vcycle}(J,\ell^J,\underline{\gamma}^J,\overline{\gamma}^J;w^J) \\-
    return $w^J$
\end{pseudo}
\caption{The FASCD full multigrid (FMG) cycle for solving VI problem \eqref{eq:fe:vi}.}
\label{alg:fascd-fmg}
\end{pseudofloat}

In contrast to the construction of LDCs (Section \ref{sec:cdmultilevel}) for V-cycles, during the ramp standard injection is used to generate the constraints $\underline{\gamma}^j,\overline{\gamma}^j$ on each level because a full solution iterate is initially needed, not just a correction.  Essentially, our choice is between \eqref{eq:fe:badstandard}, where the coarse constraint set is nonempty but then the prolonged solution may not be admissible, or \eqref{eq:fe:badmonotone}, where the constraint set may be empty.  We choose \eqref{eq:fe:badstandard}, with nonempty sets, but we must project (truncate) pointwise to the bounds, potentially adding high-frequency noise to the initial iterate.

A large space of smoothers could be explored, but in the next Section we will only apply an reduced-space Newton method \cite{Balayetal2023,BensonMunson2006}, denoted as \emph{RS Newton}.  A fixed number of preconditioned Krylov iterations approximately solve the arising linear systems, thus $O(m_j)$ work is done per smoother application.  For the symmetric Laplacian (Example \ref{ex:results:classical}) we apply conjugate gradients (CG) with incomplete Cholesky (IC) preconditioning, but otherwise generalized minimum residual (GMRES) with incomplete LU (ILU) preconditioning is used.  The coarse solver is also RS Newton, iterated to convergence using the MUMPS sparse direct linear solver~\cite{Amestoy2001}.

\section{Results} \label{sec:results}

For $P_1$ and $Q_1$ elements, FASCD Algorithms \ref{alg:fascd} and \ref{alg:fascd-fmg} were implemented in Python using the Firedrake library \cite{Rathgeberetal2016}.  The numerical results in this Section, which follow the Section \ref{sec:vi} examples, focus on the scaling of iterations with mesh resolution.  We see that each V-cycle greatly reduces the residual norm, and that a few FMG iterations solve most of these problems to within discretization error.

\begin{example} \label{ex:results:classical}
Consider the classical unilateral 2D Laplacian obstacle problem, namely Example \ref{ex:plaplacian} with $p=2$ and admissible set $\mathcal{K} = \{v \ge \psi\}$.  We consider two cases, over square domains $\Omega$, with different coincidence (active) sets $\{x\in\Omega \,:\, u(x)=\psi(x)\}$ (Figure \ref{fig:results:classical}).  The standard ``ball'' problem has an upper hemisphere obstacle \cite[Chapter 12]{Bueler2021}, but the ``spiral'' problem's obstacle is constructed to produce the coincidence set shown \cite[problem 7.1.1]{GraeserKornhuber2009}.  For both cases we used a regular coarsest mesh with $m_0=41$ nodes (64 triangles), and the finest mesh had $m_{10} \approx 3.4 \times 10^7$ nodes.  A single RS Newton iteration with three IC-preconditioned CG iterations was used as the smoother, and the initial iterate was $w=\max\{0,\psi\}$.

Table \ref{tab:results:classical} shows the resulting iteration counts.  For a first experiment we over-solve, with \texttt{atol} $=$ \texttt{rtol} $=$ \texttt{stol} $= 10^{-12}$ in stopping criterion \eqref{eq:stoppingcriterion}, to facilitate comparison to results in \cite[section 7]{GraeserKornhuber2009}.  As expected, RS Newton as a solver shows rapidly-increasing iterations, roughly proportional to resolution.  By contrast, V-cycle iterations increase slowly, perhaps logarithmically with the resolution.  For the spiral problem, \cite[Figure 7.11]{GraeserKornhuber2009} reports that a $m_J \approx 5.2 \times 10^5$ problem needs approximately 25 V-cycle iterations for the best (truncated;  ``TNMG'' and ``TMMG'') monotone multilevel algorithms.\footnote{The stopping criterion and asymptotic convergence rates in \cite{GraeserKornhuber2009} are based on norm differences from a last numerical iterate with unknown accuracy.  Our method directly measures satisfaction of the VI, via the residual norm $\|\rSS(w)\|_2$.}  The 12 V-cycle iterations needed here, at comparable resolution (8 levels), represent somewhat more arithmetic because our smoother is RS Newton using CG+IC, rather than the simpler projected Gauss-Seidel iteration.  Resolution-dependent asymptotic convergence rates for V-cycles, defined as $\rho_J=(\|\rSS(w^{J,k})\|_2/\|\rSS(w^{J,0})\|_2)^{1/k}$ where $w^{J,k}$ is the first iterate to satisfy \eqref{eq:stoppingcriterion}, are shown in Figure \ref{fig:results:asymp}.  Both iterations and rates are similar for the ball and spiral problems; the geometric complexity of the coincidence set and free boundary minimally influences FASCD V-cycle performance.  Contrast the rate of about $0.1$ for 8 levels here with about $0.4$ reported by \cite[Figure 7.11]{GraeserKornhuber2009} for TNMG/TMMG at comparable resolution.  These rates are all excellent; our better rate mostly reflects a stronger smoother.

For FMG we instead used the defaults \texttt{rtol} $=$ \texttt{stol} $= 10^{-8}$, sufficient to reach discretization error, and in Table \ref{tab:results:classical} we report the number of V-cycles after the ramp; a small, constant number solves the problem.  Since each V-cycle does the work of about eight finest-level Krylov iterations, the ball results show nearly textbook multigrid efficiency at the highest resolutions despite the low solution regularity.  More V-cycles are needed to clean-up error remaining after the ramp in the spiral problem.
\end{example}

%REGENERATE using git@bitbucket.org:pefarrell/fascd.git, in examples/
%$ python3 drape.py -prob ball -levs 9 -o ball -fascd_rtol 1.0e-6 -fascd_atol 1.0e-6 -fascd_cycle_type full
%$ mpiexec -n 16 python3 drape.py -prob spiral -levs 10 -o spiral.pvd -fascd_levels_snes_converged_reason -fascd_rtol 1.0e-6 -fascd_atol 1.0e-6 -fascd_cycle_type full
%$ paraview ball|spiral.pvd
%use preset "xray" and invert colors; for spiral.pvd, show lgap = u - lb as black if above 0.02; for ball i used 0.001
\begin{figure}[ht]
\centering
\includegraphics[width=0.35\textwidth]{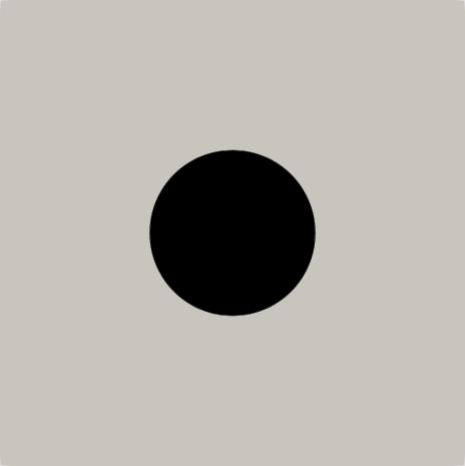} \qquad \includegraphics[width=0.35\textwidth]{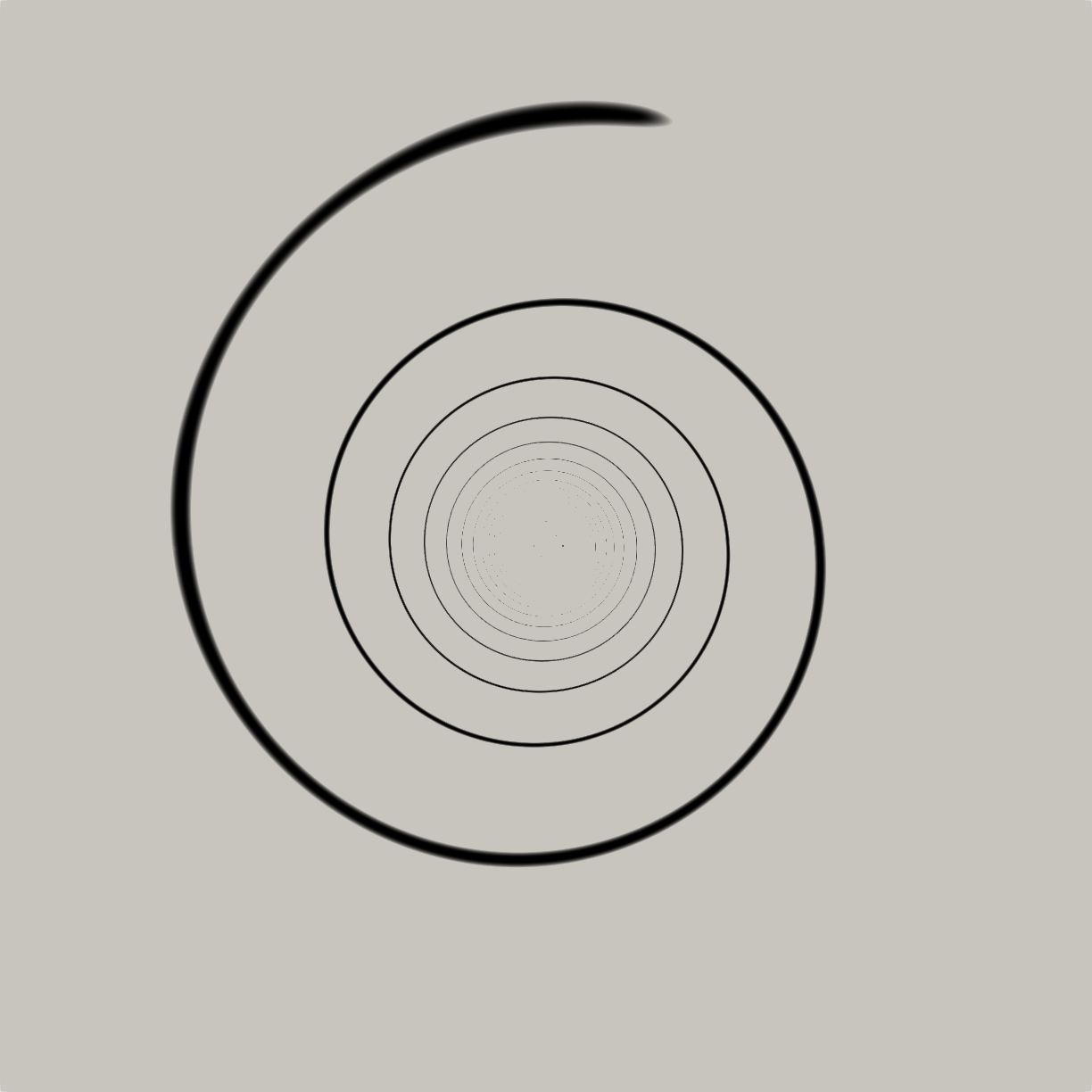}
\caption{Coincidence sets for the ball and spiral solutions in Example \ref{ex:results:classical}.}
\label{fig:results:classical}
\end{figure}

% see git@bitbucket.org:pefarrell/fascd.git: examples/drape.py
% and https://github.com/bueler/mcd-extended: paper/results/drape_{nofas,fascd}.sh|txt
\begin{table}[ht]
\centering
\begin{tabular}{lr@{\hskip 7mm}cccccc}
\toprule
\multirow{2}{*}{\emph{levels}} & \multirow{2}{*}{$m_J$} & \multicolumn{3}{c}{\,\emph{ball}} & \multicolumn{3}{c}{\,\emph{spiral}} \\ \cmidrule(lr){3-5} \cmidrule(lr){6-8}
   &                  & RS only & V-cycle & FMG & RS only & V-cycle & FMG \\
\midrule
 1 &             $41$ &   1 &  1 &  1 &   2 &  1 &  1 \\
 2 &            $145$ &   1 &  3 &  2 &   4 &  4 &  2 \\
 3 &            $545$ &   3 &  6 &  3 &   6 &  5 &  3 \\
 4 & $2.1\times 10^3$ &   6 &  7 &  3 &   7 &  7 &  4 \\
 5 & $8.3\times 10^3$ &  13 &  9 &  4 &  11 &  9 &  4 \\
 6 & $3.3\times 10^4$ &  23 & 11 &  5 &  20 & 10 &  5 \\
 7 & $1.3\times 10^5$ &  46 & 11 &  4 &  35 & 11 &  5 \\
 8 & $5.3\times 10^5$ &  90 & 12 &  4 &  67 & 12 &  5 \\
 9 & $2.1\times 10^6$ & \XX & 13 &  3 & \XX & 13 &  6 \\
10 & $8.4\times 10^6$ & \XX & 13 &  2 & \XX & 14 &  6 \\
11 & $3.4\times 10^7$ & \XX & 14 &  2 & \XX & 18 &  6 \\
\bottomrule
\end{tabular}
\bigskip
\caption{Iterations for FASCD V-cycle and FMG algorithms in Example \ref{ex:results:classical}; see text regarding tolerances.  Runs marked \XX\xspace were not attempted.}
\label{tab:results:classical}
\end{table}

%REGENERATE using results/asymprates.py
\begin{figure}[ht]
\centering
\includegraphics[width=0.65\textwidth]{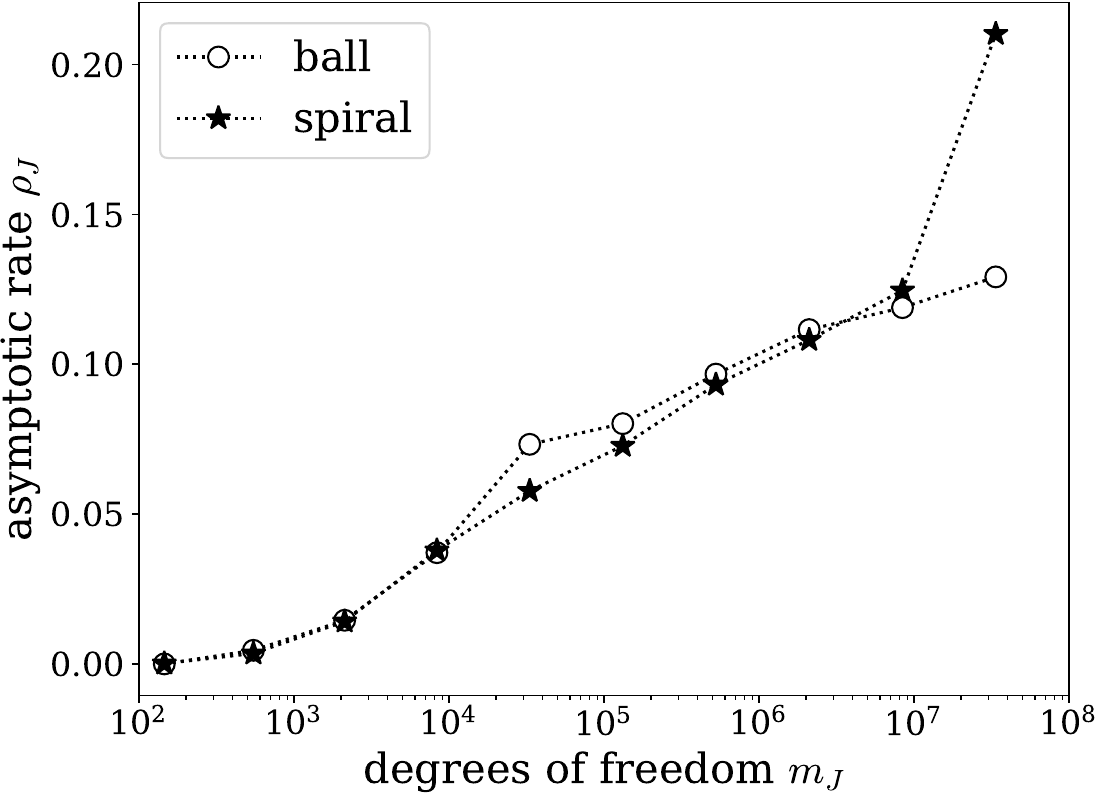}
\caption{V-cycle asymptotic convergence rates in Example \ref{ex:results:classical}.}
\label{fig:results:asymp}
\end{figure}

\begin{example}  \label{ex:results:plap}
We again consider Example \ref{ex:plaplacian}, with a unilateral constraint, but now for a nonlinear $p$-Laplacian \cite{ChoeLewis1991} in one dimension.  Setting $p=1.5$ yields a fast diffusion ($D=|\grad u|^{p-2}\to \infty$ as $|\grad u|\to 0$).  Note that Example \ref{ex:results:sia} below, by constrast, has a doubly-nonlinear operator with $p=4$ degenerate diffusion.

This example reveals the strong contrast between $V(1,0)$ and $V(0,1)$ results (\pr{up}, \pr{down} $=0$ in Algorithm \ref{alg:fascd}, respectively).  Set $\Omega=(-3,3)$, $\psi(x) = -0.2|x|$, $\mathcal{K} = \{v \ge \psi\} \subset W^{1,p}(\Omega)$, and $\ip\ell v = \int_\Omega g v$, where $g(x)=1$ for $|x|<1$ and $g(x)=-1$ for $|x|>1$, in VI \eqref{eq:vi}.  The exact solution is easily computed for any $p$; see the codes accompanying this paper.  For this experiment the coarsest mesh has $m_0=7$ nodes and the finest $m_9=3073$.  The default smoother uses three iterations of RS Newton, with LU soluton of the linear systems.  We set \texttt{rtol} $= 10^{-6}$ and \texttt{atol} $=$ \texttt{stol} $= 10^{-12}$.  Figure \ref{fig:imagesvcycle} illustrates the major steps of a V-cycle, showing each of the coarse-correction VI subproblems and solutions; the result $\tilde w^2$ is visually close to the exact solution.

% figures in the following visualization were generated using branch elbueler/simplefigs
% of git@bitbucket.org:pefarrell/fascd.git, and the following commands (including imagemagick):
%   cd examples/
%   python3 plap1d.py -fascd_cycle_type multiplicative -mcoarse 4 -levs 3 -figures
%   [save the images into mcd-extended/paper/fixfigs/vcycle/]
%   for X in start2 down2 down1 up0 up1 up2 next2; do mogrify -trim $X.png $X.png; done
\begin{figure}[ht]
\centering
\input{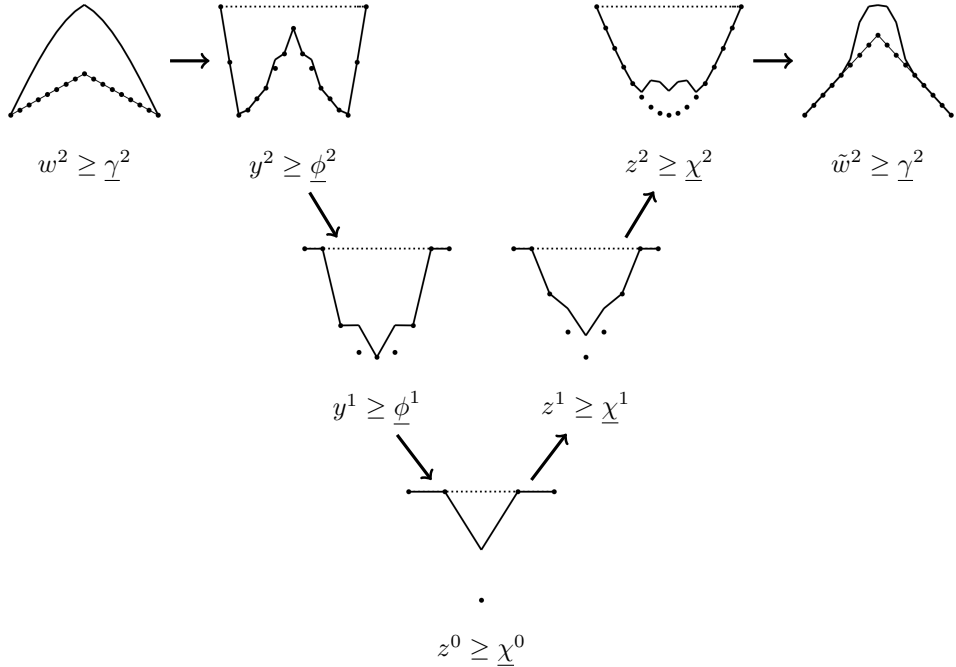}
\caption{Visualization of one FASCD V-cycle iteration for $p$-Laplacian Example \ref{ex:results:plap}, with $J=3$ and a coarsest mesh of $m_0=5$ vertices.  Solutions (i.e.~$w^j$, $y^j$, $z^j$) are lines, obstacles ($\underline{\gamma}^j$, $\underline{\phi}^j$, $\underline{\chi}^j$) are strong dots, and the zero function is dotted.  Vertical scale varies.}
\label{fig:imagesvcycle}
\end{figure}

Iteration results are shown in Table \ref{tab:results:fastplap1d}.  We see that $V(1,0)$ cycles, Algorithm 4.2 in \cite{GraeserKornhuber2009} but with a stronger smoother, often do not converge.  By contrast, $V(0,1)$ cycles, up-smoothing in the larger sets $\mathcal{U}^j$ (Section \ref{sec:cdmultilevel}), are as good as the default $V(1,1)$ cycles (Algorithm \ref{alg:fascd}).  FMG is even more efficient; a single iteration suffices at all higher resolutions.   Note that convergence is at rate $O(h^{2.0})$; compare \cite{ChoeLewis1991}.
\end{example}

% see git@bitbucket.org:pefarrell/fascd.git: examples/plap1d.py
% and https://github.com/bueler/mcd-extended: paper/results/plap1d.sh|txt
\begin{table}[ht]
\centering
\begin{tabular}{lr@{\hskip 7mm}c@{\hskip 3mm}c@{\hskip 3mm}c@{\hskip 4mm}c@{\hskip 6mm}c}
\toprule
\emph{levels} & $m_J$ & V(1,0) & V(0,1) & V(1,1) & FMG & $\|$error$\|_\infty$ \\
\midrule
 1 &    $7$ &  1 &  1 &  1 &  1 & $2.3\times 10^{-1}$ \\
 2 &   $13$ &  3 &  3 &  2 &  2 & $3.3 \times 10^{-2}$ \\
 3 &   $25$ & 12 &  4 &  2 &  1 & $9.1 \times 10^{-3}$ \\
 4 &   $49$ & NC &  4 &  2 &  1 & $3.2 \times 10^{-3}$ \\
 5 &   $97$ & NC &  3 &  3 &  1 & $5.5 \times 10^{-4}$ \\
 6 &  $193$ & NC &  3 &  3 &  1 & $1.7 \times 10^{-4}$ \\
 7 &  $385$ & NC &  3 &  3 &  1 & $4.7 \times 10^{-5}$ \\
 8 &  $769$ & NC &  4 &  7 &  1 & $9.5 \times 10^{-6}$ \\
 9 & $1537$ & NC &  5 &  6 &  1 & $3.6 \times 10^{-6}$ \\
10 & $3073$ & NC & 11 & 19 &  1 & $4.9 \times 10^{-7}$ \\
\bottomrule
\end{tabular}
\bigskip
\caption{Iterations for V-cycles and FMG in Example \ref{ex:results:plap}.  NC $=$ did not converge after 50 iterations.}
\label{tab:results:fastplap1d}
\end{table}

Firedrake applications can exploit the PETSc \cite{Balayetal2023} library's parallel solvers.  The above Examples, which were run in serial to fix the smoother, also run efficiently in parallel.  The next Examples apply parallel solvers and investigate the weak scaling of FASCD FMG.

\begin{example}  \label{ex:results:advdiff}
Consider the advection-diffusion VI problem in Example \ref{ex:advectiondiffusion} with $\Omega=(-1,1)^d$, $d=2,3$, diffusivity $\eps=0.1$ in \eqref{eq:advectiondiffusionstrong}, and zero Dirichlet boundary conditions.  The 2D problem satisfies the coercivity hypothesis, with a divergence-free velocity field $\bX = (7+5y,-5x)$, but in 3D we set $\bX = (7+5y,-5x,2z)$ so $\grad\cdot\bX \ne 0$.  The discontinuous source term $\phi \in L^\infty(\Omega)$ satisfies $\phi\ge 0$ in half of the domain, with positive circular patches, but values are negative in the other half; details are in the associated code.  While the other Examples in this Section are unilateral, here the solution is box-constrained: $0 \le u \le 1$.  Figure \ref{fig:results:advdiff} shows coincidence sets for the 2D problem; the 3D sets are comparable in regularity and general pattern.

% REGENERATE using git@bitbucket.org:pefarrell/fascd.git, in examples/
%   mpiexec -n 16 --map-by core --bind-to hwthread python3 pollutant.py -dim 2 -levs 7 -o poll2d.pvd
%   paraview poll2d.pvd
% to plot coincidence sets use ranges 0 <= u <= 0.00001 and 0.99999 <= u <= 1 and
% preset "xray" and invert colors (for u=0) or keep colors (u=1)
\begin{figure}[ht]
\centering
\includegraphics[width=0.35\textwidth]{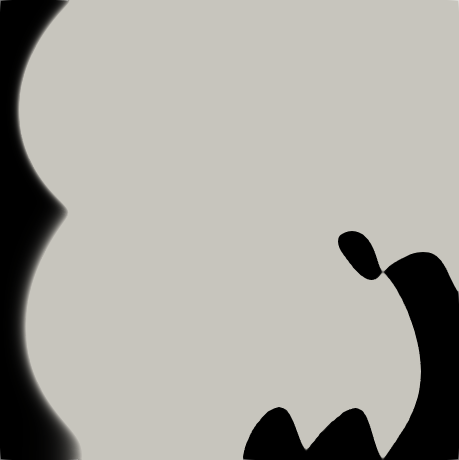} \qquad \includegraphics[width=0.35\textwidth]{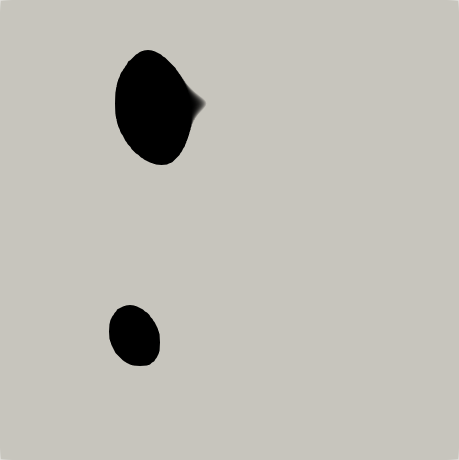}
\caption{Coincidence sets for the 2D results in Example \ref{ex:results:advdiff}: $u=0$ (left) and $u=1$ (right).}
\label{fig:results:advdiff}
\end{figure}

The 2D solver applied $P_1$ elements on a single core, but in 3D the elements were chosen to be $Q_1$ (hexahedra) on $16$ cores.  A coarse mesh with $m_0=16^d$ vertices minimally-resolved the coincidence set.  The local Péclet number $\text{Pe}=\|\bX(x,y,z)\|_2 L \eps^{-1}$, with $L=2$ the width of $\Omega$, varied from 40 to 260, and the mesh Péclet number $m_J^{-1/d}\,\text{Pe}$ varied further with resolution, so the $m_J\le 121^d$ problems were advection-dominated, and those with $m_J\ge 481^d$ diffusion-dominated.  However, the smoother was the same in all cases: two (fixed) RS Newton iterations, with three preconditioned (additive Schwarz method (ASM) with ILU on each process) GMRES iterations.  Tolerances were set to \id{rtol} $=10^{-5}$ and \id{atol} $=$ \id{stol} $= 10^{-9}$.  FMG iteration results are shown in Table \ref{tab:results:advdiff}; a few iterations sufficed up to these memory-limited resolutions.
\end{example}

% see git@bitbucket.org:pefarrell/fascd.git: examples/pollutant.py
% and https://github.com/bueler/mcd-extended: paper/results/pollutant.sh|txt
\begin{table}[ht]
\centering
\begin{tabular}{lr@{\hskip 7mm}c@{\hskip 4mm}c}
\toprule
\emph{levels} & $m_J$ & $d=2$ & $d=3$ \\
\midrule
 1 &    $16^d$ & 1 & 1 \\
 2 &    $31^d$ & 2 & 2 \\
 3 &    $61^d$ & 2 & 2 \\
 4 &   $121^d$ & 2 & 2 \\
 5 &   $241^d$ & 2 & 3 \\
 6 &   $481^d$ & 2 \\
 7 &   $961^d$ & 3 \\
 8 &  $1921^d$ & 3 \\
 9 &  $3841^d$ & 3 \\
\bottomrule
\end{tabular}
\bigskip
\caption{FMG iterations for 2D ($P_1$, serial) and 3D ($Q_1$, 16 processes) solutions of the advection-diffusion bi-obstacle VI problem in Example \ref{ex:results:advdiff}.}
\label{tab:results:advdiff}
\end{table}

\begin{example}   \label{ex:results:sia}
Consider the steady shallow ice approximation (SIA) obstacle problem (Example \ref{ex:sia}), in which the surface elevation is constrained to exceed the bedrock elevation, $s\ge b$, on $\Omega=(0,L)^2$ with $L=1800$ km.  The data consist of a radial surface mass balance function $a$ \cite[equation (5.122)]{GreveBlatter2009} and a bumpy, synthetic map $b\in C^1(\Omega)$, with an elevation range of $1800$ m; details are in the associated code.  A numerical solution is shown in Figure \ref{fig:results:siascene}, colored by the solution surface flow speed, the magnitude of $\mathbf{u} = (n+1)(n+2)^{-1} \Gamma (s-b)^{n+1} |\grad s|^{n-1} \grad s$.  Note $|\grad s|$ is unbounded at the free boundary; the solution $s$, while continuous, has very low regularity.

% to regenerate from git@bitbucket.org:pefarrell/fascd.git: examples/:
%   $ mpiexec -n 16 --map-by core --bind-to hwthread python3 sia.py -levs 8 -bumps \
%       -mcoarse 8 -fascd_cycle_type full -fascd_monitor -fascd_rtol 2.0e-4 \
%       -fascd_atol 1.0e-8 -o sialev8.pvd
\begin{figure}[ht]
\centering
\includegraphics[width=1.0\textwidth]{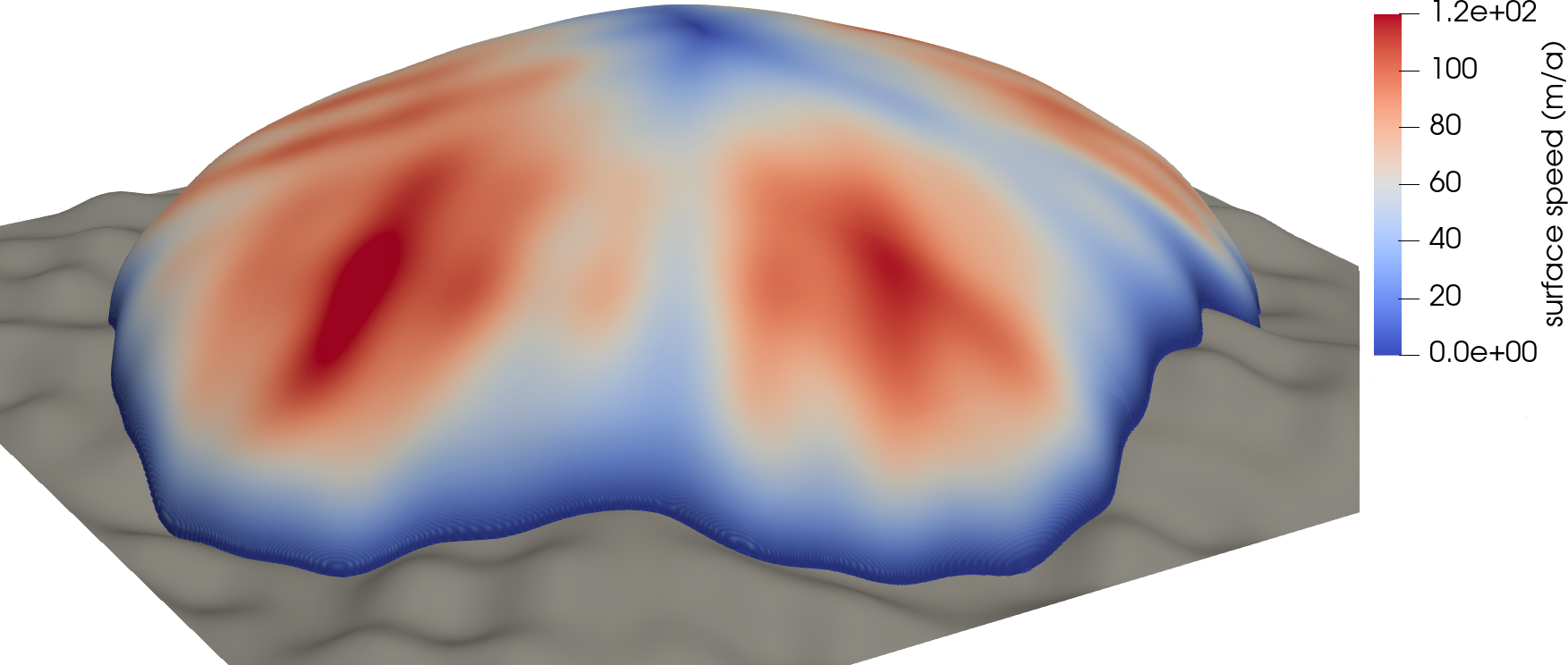}
\caption{Example \ref{ex:results:sia} solution over a bumpy bedrock topography (grey), with 100-times vertical exaggeration and surface coloring by flow speed.}
\label{fig:results:siascene}
\end{figure}

We applied FASCD FMG using $Q_1$ elements and a smoother of four (fixed) iterations of RS Newton (Section \ref{sec:implementation}), with quadratic backtracking line search.  The arising linear systems were approximately solved by three (fixed) iterations of preconditioned GMRES and an ASM+ILU preconditioner.  The coarse mesh was relatively fine, with $m_0=21^2$ vertices and 400 cells.  Because of the low regularity of the solution we used \id{rtol} $= 2 \times 10^{-4}$ and \id{atol} $= 10^{-8}$.  The initial iterate was zero ice: $s=b$.  These choices generated a solver with mesh-independent iterations.

For a weak-scaling study we fixed the number of degrees of freedom per process at $m_J/P=641^2 \approx 4.1 \times 10^5$ and solved with $6, \dots, 11$ levels.\footnote{The experiments were conducted on ARCHER2, the UK national supercomputer. Each node has 16 memory channels, so a maximum of 16 cores per node were employed.}  The results in Table \ref{tab:results:siaweak} show some degradation in time per cycle with increasing core counts, perhaps due to the use of one-level additive Schwarz in the Krylov iteration.  Nevertheless, the finest mesh result, with $m_J=20481^2=4.1 \times 10^8$ degrees of freedom, required less than three minutes to solve for the steady surface elevation of an ice sheet larger than the current Greenland ice sheet ($1.7$ million $\text{km}^2$) at uniform $87.5$ m resolution.

% see git@bitbucket.org:pefarrell/fascd.git: examples/sia.py
% and https://github.com/bueler/mcd-extended: paper/results/siaweak.sh|txt
\begin{table}[ht]
\centering
\begin{tabular}{c@{\hskip 4mm}c@{\hskip 4mm}c@{\hskip 7mm}c@{\hskip 4mm}cc}
\toprule
$P$ & \emph{levels} & $m_J$ & \emph{iterations} & \emph{time} (s) & \emph{time$/$iteration} (s) \\
\midrule
 1 & 6 & $641^2$ & 3 & 98 & 32.7 \\
 4 & 7 & $1281^2$ & 3 & 98 & 32.7 \\
 16 & 8 & $2561^2$ & 3 & 124 & 41.3 \\
 64 & 9 & $5121^2$ & 3 & 136 & 45.3 \\
 256 & 10 & $10241^2$ & 2 & 95 & 47.5 \\
 1024 & 11 & $20481^2$ & 2 & 177 & 88.5 \\
 \bottomrule
\end{tabular}
\bigskip
\caption{Iterations and wall clock time for parallel FASCD FMG solutions of Example \ref{ex:results:sia} on $P$ processes.}
\label{tab:results:siaweak}
\end{table}
\end{example}

\section{Discussion and outlook} \label{sec:discussion}

As a multilevel solver for VI problems, the FASCD method is a strategy for generating coarser-level problems which, once solved, make helpful admissible additions to a V-cycle update.  The method is smoother-agnostic, and our implementation within the extensible Firedrake \cite{Rathgeberetal2016} and PETSc \cite{Balayetal2023} library framework will allow easy experimentation with other smoothers (projected Gauss--Seidel, semi-smooth Newton, interior-point, etc.), beyond the active-set Newton method chosen for demonstration here.  For smoothers of Newton-Krylov type one could also add linear geometric or algebraic multigrid preconditioning \cite{Trottenbergetal2001} for the linear step problems, instead of our simpler choice of incomplete factorization preconditioners.  In a different direction, a restructured multilevel CD approach which applies the additive/parallel CD iteration \cite{Tai2003} might be considered, in pursuit of a VI algorithm analogous to the BPX multigrid method for PDEs \cite{BramblePasciakXu1990}.  In any case, the performance observed in Section \ref{sec:results} is already excellent on the tested examples.

\subsection*{Code availability} \label{sec:code}  The software used to produce the results is archived at tag \texttt{v1.0} in the repository \texttt{\url{https://bitbucket.org/pefarrell/fascd/}}, using the Firedrake version \texttt{3fb16ad47}.

\bibliographystyle{siamplain}
\bibliography{fascd}

\end{document}